\documentclass[10pt]{article}
\usepackage[top=3.5cm, bottom=3cm, left=3cm, right=2.5cm]{geometry} 

\usepackage{amssymb,amsmath,amsfonts, amsthm}
\numberwithin{equation}{section}
\usepackage{mathtools}

\usepackage{authblk}
\usepackage{hyperref}
\usepackage{cleveref}

\usepackage[shortlabels]{enumitem} 
\usepackage{graphicx}

\usepackage{mathrsfs}
\usepackage{bm}

\usepackage[backend=biber]{biblatex}
\theoremstyle{plain}
\newtheorem{theorem}{Theorem}[section]
\theoremstyle{plain}
\newtheorem{lemma}[theorem]{Lemma}
\newtheorem{proposition}[theorem]{Proposition}
\newtheorem{corollary}[theorem]{Corollary}
\theoremstyle{definition}
\newtheorem{definition}[theorem]{Definition}
\newtheorem{remark}[theorem]{Remark}
\newtheorem{assumption}[theorem]{Assumption}

\AddToHook{env/lemma/begin}{\crefalias{theorem}{lemma}}
\AddToHook{env/proposition/begin}{\crefalias{theorem}{proposition}}
\AddToHook{env/corollary/begin}{\crefalias{theorem}{corollary}}
\AddToHook{env/definition/begin}{\crefalias{theorem}{definition}}
\AddToHook{env/remark/begin}{\crefalias{theorem}{remark}}
\AddToHook{env/assumption/begin}{\crefalias{theorem}{assumption}}
\AddToHook{env/algo/begin}{\crefalias{theorem}{algo}}
\AddToHook{env/ex/begin}{\crefalias{theorem}{ex}}
\crefname{assumption}{assumption}{assumptions}
\Crefname{assumption}{Assumption}{Assumptions}

\newcounter{ass}

\newcommand{\eps}{\varepsilon}

\newcommand{\cD}{\mathcal{D}}

\newcommand{\cF}{\mathcal{F}}
\newcommand{\cG}{\mathcal{G}}

\newcommand{\cI}{\mathcal{I}}

\newcommand{\cO}{\mathcal{O}}

\newcommand{\cS}{\mathcal{S}}

\newcommand{\bC}{\mathbb{C}}

\newcommand{\bN}{\mathbb{N}}

\newcommand{\bR}{\mathbb{R}}

\newcommand{\bZ}{\mathbb{Z}}

\newcommand{\mfv}{\mathfrak{v}}

\newcommand{\mfu}{\mathfrak{u}}
\newcommand{\bgg}{\bm{\gamma}}
\newcommand{\mfs}{\mathfrak{s}}

\newcommand{\floor}[1]{\left\lfloor #1 \right\rfloor}
\newcommand{\ceil}[1]{\left\lceil #1 \right\rceil}
\newcommand\dif{\mathop{}\!\mathrm{d}}

\DeclarePairedDelimiterX{\inp}[2]{\langle}{\rangle}{#1, #2}
\DeclarePairedDelimiter\abs{\lvert}{\rvert}
\DeclarePairedDelimiter\norm{\lVert}{\rVert}

\DeclareMathOperator*{\supp}{supp}

\crefname{ass}{assumption}{assumptions}
\Crefname{ass}{Assumption}{Assumptions}

\usepackage{xcolor}

\title{
    Higher-Order QMC $hp$-FEM UQ for Singular Perturbation:\\
    Reaction--Diffusion with Gevrey-Parametric Reaction Coefficient
}
\date{September 7, 2026}

\author[$\dagger$]{F. Rohner}
\author[$\dagger$]{Ch. Schwab} 

\affil[$\dagger$]{
    \footnotesize Seminar for Applied Mathematics, ETH Z\"{u}rich\\
    R\"{a}mistrasse 101, 8092 Z\"urich, Switzerland\\
    \texttt{\char`\{fabian.rohner,christoph.schwab\char`\}@sam.math.ethz.ch}
}

\begin{document}
\maketitle

{\small
\begin{abstract}
We analyze the combined discretization of singularly perturbed
reaction--diffusion problems with perturbation parameter $\eps\in (0,1]$
and with parametric, Gevrey-regular reaction term,
by $hp$-Galerkin FE methods and by higher-order polynomial lattice rules. 
We establish $\eps$-robust algebraic convergence rates 
for the combined QMC $hp$-FEM approximation
that are not subject to the curse of dimensionality (CoD) in the parametric reaction coefficient,
and only limited by the sparsity and the $\mfs_{y}$-Gevrey regularity 
of the reaction coefficient expansion.
On the interval $(-1,1)$, we establish an $\eps$-robust,
root-exponential convergence rate $\exp(-\sigma \, p_{\mathrm{FE}}^{1/\mfs_{x}})$ 
for $hp$-Galerkin FEM for reaction--diffusion problems under $\mfs_{x}$-Gevrey
regularity assumptions on the reaction coefficient and the source term, 
a result that can be of independent interest.
Numerical experiments are consistent with the theoretical bounds.
\end{abstract}
}

\section{Introduction}
\label{sec:Intro}
The efficient numerical approximation of mathematical expectations of (functionals of)
solutions of PDEs with parametric or stochastic input data (e.g. loading, coefficients, domains)
is a core area of numerical uncertainty quantification (UQ) 
that has attracted considerable attention in recent years. 
The numerical analysis of efficient solution methods has 
generally focused on combining convergence rate bounds for classical 
discretization schemes for the PDE under consideration 
with error bounds for sampling methods, 
such as Monte Carlo (MC), Quasi--Monte Carlo (QMC) or sparse grid (Sp G) 
resp. ``stochastic collocation'' methods. 
In most, if not all, references, it was assumed that the
PDE problem of interest is \emph{regular}, 
i.e. that there holds unique solvability and 
that error bounds on the numerical PDE solution 
hold uniformly over the set of admissible PDE input data.
Often, in applications, 
PDEs are \emph{singularly perturbed}. 
I.e. they depend, in addition to the uncertain input data, also on one or several small 
\emph{physical, non-dimensional} parameters.
Then, 
the interest lies in \emph{robust numerical approximation of solutions in the small parameter regime}. 
Prototypical examples of this scenario are 
elliptic reaction--diffusion equations with small diffusion coefficients.
It has been well-known since the fundamental works \cite{Bakhvalov1969,Shishkin1989} 
that \emph{robust} PDE discretizations minimally require 
\emph{local mesh refinement for the robust solution approximation}, 
see e.g. \cite{RST2nd} and the references there.
The numerical analysis of higher order UQ methods in the presence of singular perturbations is 
the purpose of the present paper.
\subsection{Previous Results}
\label{sec:PrRes}
The use of Higher-Order Quasi--Monte Carlo (HoQMC) integration for the efficient
approximation of mathematical expectations of solutions of PDEs with 
parametric inputs has received attention in the past decade
\cite{dick2014higher,dick2016holomorphic,dick2016multilevel,kuo2016survey}.

Generally, in computational PDE UQ
the use of HoQMC quadratures is preceded by a so-called
``dimension-truncation'' to PDE inputs that depend on 
a finite number of parameters. 
We refer to \cite{guth2024lognormal} and the references there.

In \cite{RS26_1180}, which is closely related to the present work, 
for the singularly perturbed reaction--diffusion problem considered
also in the present work,
we analyzed the use of a High-Order QMC quadrature combined
with a standard, first-order Lagrange finite element discretization on
boundary-layer resolving Shishkin grids.
In \cite{RS26_1180}, we assumed 
\emph{affine-parametric coefficients with limited (low) Sobolev regularity}. 
As shown in \cite{RS26_1180}, the parametric solutions then depend analytically
on the parameters in the reaction coefficient, \emph{with uniform
    w.r. to the singular perturbation parameter sizes of the domains of analyticity}. 
This, in turn, allowed us to show in \cite{RS26_1180}
robust convergence of 
first order for the FE discretization and 
convergence rates for the HoQMC quadratures which are 
free from the so-called ``curse of dimensionality'' (CoD)
depending only on the summability of 
the affine-parametric expansion of the reaction coefficient.

In \cite{chernov2025analytic}, parametric regularity under Gevrey-assumptions
on parameters was investigated for regular linear and semilinear elliptic PDEs.
Convergence rates for Quasi--Monte Carlo quadratures 
for numerical estimation of statistical moments were analyzed.

In \cite{ScSu96,melenk1997boundarylayer}, 
so-called $hp$-Finite Element methods for singular perturbation problems were analyzed. 
\emph{Robust exponential convergence rates} were shown under analyticity
assumptions on the data (source term, reaction coefficient).
\subsection{Contributions}
\label{sec:Contr}
While in \cite{ManuskriptDLBL} and \cite{RS26_1180},
we considered low-order FE discretization in 
the physical domain $D = (-1,1)$ with particular attention to low regularity
data and solutions, we address here a \emph{Gevrey-regular reaction coefficient}
which is parametric on a possibly high-dimensional parameter space.
Accordingly, we combine high-order Quasi--Monte Carlo discretization in the parameter
domain with boundary layer $hp$-FE Galerkin discretizations in the physical domain.
A novel approach via a $\cG^{\mfs_{y}}$-convolution identity is introduced to ascertain 
the $\eps$-explicit $\cG^{\mfs_{y}}$-regular dependence of the solution on the parameters.
A necessary technical ingredient in the $hp$-FE error analysis of independent interest
is the generalization
of the robust exponential convergence rate of $hp$-FEM 
which was proved in \cite{ScSu96,melenk1997boundarylayer} to the more general setting
of $\cG^{\mfs_{x}}$-regular reaction coefficients. In brief, in the present paper we show
1. robust, root-exponential convergence rates of $hp$-FEM for singularly perturbed, 
elliptic reaction--diffusion problems under Gevrey-$\mfs_{x}$ regularity, 
and 2. dimension-independent convergence of order $N^{-1/p}$ for higher-order QMC
approximation of expectations of quantities of interest (QoIs), 
given by bounded linear functionals of the parametric solution, 
provided that the anisotropy weights satisfy
$\bm{\beta} \in \ell^p(\bN)$, uniformly with respect to
$\eps$.
If $\bm{\beta}\in\bigcap_{p>0}\ell^p(\bN)$, as for weights with Gevrey-type decay, 
this yields super-algebraic convergence in the sense that any prescribed
algebraic order can be achieved by a suitable higher-order QMC rule; the
construction and its constants may depend on the selected order.
For certain affine Fourier/KL-type parametrizations, 
such as the model considered in the numerical experiments, 
this weight decay encodes the spatial Gevrey regularity of the coefficient.
\subsection{Layout}
The structure of this text is as follows.
Upon fixing some notation in Sec.~\ref{sec:Notat}, in 
Sec.~\ref{sec:param-reg} we present results on $\eps$-explicit
$\cG{\mfs_{y}}$ parametric regularity of solutions, 
generalizing \cite{RS26_1180} to a $\cG^{\mfs_{y}}$-setting.
Sec.~\ref{sec:dim-trunc} addresses the ``dimension-truncation'' 
error from an infinite number of parameters to a finite number,
again in an $\eps$-explicit manner.
Sec.~\ref{sec:qmc} reviews (known) results on higher-order Quasi--Monte
Carlo quadrature, in the specific context of $\eps$-dependent
weighted norm bounds on the solution.
Sec.~\ref{sec:combined-error} then presents an $\eps$-explicit error bound
of the combined $hp$-FE-HoQMC discretization of the parametric singularly
perturbed problem.
Sec.~\ref{sec:numerical-experiments} 
contains numerical experiments that support the theoretical results.
Appendix~\ref{sec:ExpCnv} contains an extension of the 
exponential convergence analysis of \cite{melenk1997boundarylayer} 
to the presently considered $\cG^{\mfs_{x}}$ setting.
\subsection{Notation}
\label{sec:Notat}
With $\bN = \{1,2,3,...\}$ we denote the natural numbers, and
write $\bN_0 := \{ 0,1,2,...\}$ and $\bN^* := \bN \cup \{ \infty \}$.
Furthermore, $\bR$ and $\bC$ denote the sets of real and complex numbers, respectively.
For $J \in \bN^*$, we write $\bN^J, \bR^J, \bC^J$ for the corresponding $J$-fold Cartesian products.
For $J=\infty$ this is understood as the countably infinite Cartesian product.
Let $\cF$ denote the countable set of finitely supported multi-indices 
$\bm\nu=(\nu_j)_{j\geq1}\in\bN_0^{\bN}$. 
For $\bm\nu\in\cF$, the following quantities are finite:
\begin{equation*}
    \supp(\bm\nu)\coloneqq\{j\in\bN:\nu_j\neq0\},
    \qquad
    \abs{\bm\nu}\coloneqq\sum_{j\geq1}\nu_j,
    \qquad
    \bm\nu!\coloneqq\prod_{j\in\supp(\bm\nu)}\nu_j!.
\end{equation*}
For $\bm\mu,\bm\nu\in\cF$, we write $\bm\mu\leq\bm\nu$ if
$\mu_j\leq\nu_j$ for every $j\geq1$, and define
\begin{equation*}
    \binom{\bm\nu}{\bm\mu}
    \coloneqq\prod_{j\in\supp(\bm\nu)\cup\supp(\bm\mu)}
    \binom{\nu_j}{\mu_j},
    \qquad
    \bm\beta^{\bm\nu}\coloneqq
    \prod_{j\in\supp(\bm\nu)}\beta_j^{\nu_j},
    \qquad \bm\beta^{\bm0}\coloneqq1,
\end{equation*}
for any nonnegative sequence $\bm\beta=(\beta_j)_{j\geq1}$. Thus all products
are finite, including when some $\beta_j=0$.

Let $X$ be a Banach space, let $\mathcal O\subset\bR^m$ be a bounded domain,
and let $\mfs\geq1$.  We write
$g\in\cG^{\mfs}(\overline{\mathcal O};X)$ if $g$ admits a $C^\infty$,
$X$-valued extension to a neighborhood of $\overline{\mathcal O}$ and there
exist constants $C_g,\rho_g \ge 0$ such that
\begin{equation}
    \label{eq:gevrey-class-finite-dimensional}
    \sup_{\bm z\in\overline{\mathcal O}}
    \norm{\partial_{\bm z}^{\bm\nu}g(\bm z)}_X
    \leq C_g\rho_g^{\abs{\bm\nu}}
    (\abs{\bm\nu}!)^{\mfs}
    \qquad\forall\bm\nu\in\bN_0^m.
\end{equation}
For scalar-valued functions, we simply write
$\cG^{\mfs}(\overline{\mathcal O})$.

Let $U=[-1/2,1/2]^{\bN}$, endowed with the product topology.  For a
nonnegative sequence $\bm\beta=(\beta_j)_{j\geq1}$, we write
$g\in\cG_{\bm\beta}^{\mfs}(U;X)$ if $g:U\to X$ is continuous, every
finite-coordinate section admits a jointly $C^\infty$, $X$-valued extension
to a neighborhood of the corresponding finite-dimensional parameter cube,
and the resulting mixed derivatives are continuous and satisfy
\begin{equation}
    \label{eq:gevrey-class-parametric}
    \sup_{\bm y\in U}
    \norm{\partial_{\bm y}^{\bm\nu}g(\bm y)}_X
    \leq C_g\rho_g^{\abs{\bm\nu}}
    (\abs{\bm\nu}!)^{\mfs}\bm\beta^{\bm\nu}
    \qquad\forall\bm\nu\in\cF
\end{equation}
for some constants $C_g,\rho_g \ge 0$.  Thus
\eqref{eq:gevrey-class-finite-dimensional} describes spatial Gevrey
regularity, whereas the weights in \eqref{eq:gevrey-class-parametric} encode
anisotropy of the parameter dependence.

\section{\texorpdfstring{$\eps$}{ε}-Robust parametric Gevrey regularity}
\label{sec:param-reg}

We establish well-posedness and mixed-parametric regularity for 
{solutions of a model linear singularly perturbed reaction-diffusion problem 
    with reaction coefficients that have general $\mfs_{y}$-Gevrey dependence on the parameter. 
}
The central point is that all stability and regularity constants are explicit in
the perturbation parameter $\eps\in(0,1]$ when the solution is measured in
the weighted norm $\norm{\cdot}_{1,\eps,D}$.

\subsection{Model problem and assumptions}
\label{subsec:model-assumptions}

Let $D\subset\bR^d$ be a bounded Lipschitz domain.  With $U = [-1/2,1/2]^{\bN}$ as in
\cref{sec:Notat}, for $\eps\in(0,1]$ and $\bm y\in U$, we consider
\begin{align}
    -\eps^2\Delta_x u_\eps(x;\bm y)+b(x;\bm y)u_\eps(x;\bm y)
    &=f(x), &&x\in D,
    \label{eq:model-pde}\\
    u_\eps(x;\bm y)&=0, &&x\in\partial D,
    \label{eq:model-bc}
\end{align}
where $f\in L^2(D)$ is independent of $\bm y$.
We shall work under the following assumptions.

\begin{assumption}[Uniform positivity and boundedness]
    \label{ass:uniform-bounds}
    There exist constants $0<b_0\leq b_1<\infty$ such that
    \begin{equation}
        \label{eq:uniform-bounds}
        b_0\leq b(x;\bm y)\leq b_1
        \qquad\text{for a.e. }x\in D\text{ and every }\bm y\in U.
    \end{equation}
    The constants $b_0$ and $b_1$ are independent of $\eps$ and $\bm y$.
\end{assumption}

\begin{assumption}[Weighted $\mfs_{y}$-Gevrey regularity]
    \label{ass:gevrey-b}
    Let $\mfs_{y}\geq1$.  For some nonnegative sequence
    $\bm\beta=(\beta_j)_{j\geq1}$, assume that
    \begin{equation*}
        b\in\cG_{\bm\beta}^{\mfs_y}(U;L^\infty(D)).
    \end{equation*}
    We fix constants $C_b>0$ and $\rho_b>0$ for which
    \begin{equation}
        \label{eq:gevrey-b-bound}
        \sup_{\bm y\in U}
        \norm{\partial_{\bm y}^{\bm\nu}b(\cdot;\bm y)}_{L^\infty(D)}
        \leq
        C_b(\abs{\bm\nu}!)^{\mfs_{y}}\rho_b^{\abs{\bm\nu}}
        \bm\beta^{\bm\nu}
        \qquad\forall\bm\nu\in\cF.
    \end{equation}
    The constants $C_b$ and $\rho_b$, and the sequence $\bm\beta$, are independent
    of $\eps$.
\end{assumption}

\begin{remark}
    The sequence $\bm\beta$ describes the anisotropy of the parameter dependence.
    No summability condition is needed for the regularity result below. Conditions
    such as $\bm\beta\in\ell^p(\bN)$ will enter later in the dimension-truncation
    and QMC analyses. 
    The factor $\rho_b$ could be absorbed into the sequence $\bm\beta$.
\end{remark}

For $\eps\in(0,1]$ and for $v\in H^1(D)$, define the $\norm{v}_{1,\eps,D}$ norm 
via
\begin{equation}
    \label{eq:energy-norm}
    \norm{v}_{1,\eps,D}^2
    \coloneqq\eps^2\norm{\nabla v}_{L^2(D)}^2+\norm{v}_{L^2(D)}^2.
\end{equation}
When the domain is clear from the context, we write $\norm{v}_{1,\eps}$ in place of $\norm{v}_{1,\eps,D}$.
We further let
\begin{equation}
    \label{eq:Veps}
    V_\eps\coloneqq\overline{C_0^\infty(D)}^{\norm{\cdot}_{1,\eps,D}}.
\end{equation}
For $\eps>0$, the space $V_\eps$ is $H_0^1(D)$ as a set, equipped with the
$\eps$-dependent norm \eqref{eq:energy-norm}. 
On the dual space $V_\eps'$, we use the $\eps$-dependent dual norm
\begin{equation}
    \label{eq:dual-norm}
    \norm{F}_{-1,\eps,D}
    \coloneqq\norm{F}_{V_\eps'}
    \coloneqq\sup_{v\in V_\eps\setminus\{0\}}
    \frac{\abs{\inp{F}{v}_{V_\eps'\times V_\eps}}}
    {\norm{v}_{1,\eps,D}}.
\end{equation}
Identifying $L^2(D)$ with its dual, 
as $\eps\downarrow 0$, 
formally both norms tend to the $\norm{\cdot}_{L^2(D)}$ norm, so that 
for $0\leq \eps \leq 1$, $V_\eps \subseteq L^2(D) \simeq L^2(D)^* \subseteq V_\eps'$,
with equality in the inclusions for $\eps = 0$.
The canonical embedding $L^2(D)\hookrightarrow V_\eps'$ has norm at most one,
uniformly in $\eps\in(0,1]$. 
Consequently,
\begin{equation}
    \label{eq:L2-dual-embedding}
    \norm{f}_{-1,\eps,D}\leq\norm{f}_{L^2(D)}.
\end{equation}

For $\bm y\in U$, define the parametric bilinear form
\begin{equation}
    \label{eq:bilinear-form}
    B_\eps(w,v;\bm y)
    \coloneqq
    \eps^2\int_D\nabla w(x)\cdot\nabla v(x)\,\dif x
    +\int_D b(x;\bm y)w(x)v(x)\,\dif x,
    \qquad w,v\in V_\eps.
\end{equation}
Then the weak formulation of \eqref{eq:model-pde}--\eqref{eq:model-bc} 
is to find $u_\eps(\cdot;\bm y)\in V_\eps$ such that
\begin{equation}
    \label{eq:weak-problem}
    B_\eps(u_\eps(\cdot;\bm y),v;\bm y)
    =\inp{f}{v}_{V_\eps'\times V_\eps}
    \qquad\forall v\in V_\eps.
\end{equation}
Since $f\in L^2(D)$, the right-hand side in \eqref{eq:weak-problem} equals
$(f,v)_{L^2(D)}$.

By \cref{ass:uniform-bounds}, the bilinear form \eqref{eq:bilinear-form} is
continuous and coercive, uniformly in $\eps$ and $\bm y$. More precisely,
\begin{align}
    \abs{B_\eps(w,v;\bm y)}
    &\leq(1+b_1)\norm{w}_{1,\eps,D}\norm{v}_{1,\eps,D},
    &&w,v\in V_\eps,
    \label{eq:continuity}\\
    B_\eps(v,v;\bm y)
    &\geq\min\{1,b_0\}\norm{v}_{1,\eps,D}^2,
    &&v\in V_\eps.
    \label{eq:coercivity}
\end{align}
The constants in \eqref{eq:continuity} and \eqref{eq:coercivity} are
independent of $\eps$ and $\bm y$. Set
\begin{equation}
    \label{eq:B-constant}
    B\coloneqq\frac{1}{\min\{1,b_0\}}.
\end{equation}
In particular, $B$ is independent of $\eps$ and $\bm y$.

\begin{proposition}[Uniform well-posedness]
    \label{prop:wellposedness}
    Under \cref{ass:uniform-bounds}, for every $\eps\in(0,1]$ and $\bm y\in U$,
    problem \eqref{eq:weak-problem} has a unique solution
    $u_\eps(\cdot;\bm y)\in V_\eps$. Moreover,
    \begin{equation}
        \label{eq:solution-bound}
        \norm{u_\eps(\cdot;\bm y)}_{1,\eps,D}
        \leq B\norm{f}_{-1,\eps,D}
        \leq B\norm{f}_{L^2(D)},
    \end{equation}
    where $B$ is the $\eps$-independent constant defined in
    \eqref{eq:B-constant}.
\end{proposition}
\begin{proof}
    The assertion follows from the Lax--Milgram theorem, using the continuity
    estimate \eqref{eq:continuity} and the coercivity estimate
    \eqref{eq:coercivity}. Testing \eqref{eq:weak-problem} with
    $v=u_\eps(\cdot;\bm y)$ gives the first inequality in
    \eqref{eq:solution-bound}; the second follows from
    \eqref{eq:L2-dual-embedding}. Neither estimate introduces a constant depending
    on $\eps$ or $\bm y$.
\end{proof}
\subsection{Equations for the mixed parametric derivatives}
\label{subsec:derivative-equations}
We first record the stability of the solution with respect to perturbations of
its reaction coefficient. Let $\bm y,\widetilde{\bm y}\in U$, and set
\begin{equation*}
    w_\eps(\cdot;\bm y,\widetilde{\bm y})
    \coloneqq u_\eps(\cdot;\bm y)-u_\eps(\cdot;\widetilde{\bm y}).
\end{equation*}
Subtracting the two weak formulations gives
\begin{equation}
    \label{eq:coefficient-stability-equation}
    \begin{split}
        B_\eps(w_\eps(\cdot;\bm y,\widetilde{\bm y}),v;\bm y)
        ={}&-\int_D\bigl(b(x;\bm y)-b(x;\widetilde{\bm y})\bigr)
        u_\eps(x;\widetilde{\bm y})v(x)\,\dif x,
        \qquad\forall v\in V_\eps.
    \end{split}
\end{equation}
Choosing $v=w_\eps(\cdot;\bm y,\widetilde{\bm y})$, using
\eqref{eq:coercivity}, and then applying \eqref{eq:solution-bound}, we obtain
\begin{equation}
    \label{eq:coefficient-stability}
    \norm{u_\eps(\cdot;\bm y)-u_\eps(\cdot;\widetilde{\bm y})}_{1,\eps,D}
    \leq B^2
    \norm{b(\cdot;\bm y)-b(\cdot;\widetilde{\bm y})}_{L^\infty(D)}
    \norm{f}_{-1,\eps,D}.
\end{equation}
Here and below, $B$ is the same constant as in \eqref{eq:B-constant}. In
particular, the stability constant in \eqref{eq:coefficient-stability} is
independent of $\eps\in(0,1]$.

For a parameter-dependent map $g$ and a coordinate direction $k\geq1$, we use
the coordinate difference quotient
\begin{equation}
    \label{eq:difference-quotient-operator}
    \delta_h^k g(\cdot;\bm y)
    \coloneqq
    \frac{g(\cdot;\bm y+h\bm e_k)-g(\cdot;\bm y)}{h},
\end{equation}
whenever $h\neq0$ is sufficiently small and
$\bm y+h\bm e_k\in U$. 
At an endpoint of a coordinate interval, 
only admissible one-sided quotients are used. 
Their limits agree with the
derivatives of the finite-dimensional extensions 
in \cref{ass:gevrey-b}.
\begin{proposition}[Differentiated variational equations]
    \label{prop:derivative-equation}
    Let $\eps\in(0,1]$ and suppose \cref{ass:uniform-bounds,ass:gevrey-b} hold.
    
    Then for every $\bm y \in U$ and for every $\bm\nu\in\cF$,
    the map $\bm y\mapsto u_\eps(\cdot;\bm y)$ 
    admits a continuous mixed derivative
    $\partial_{\bm y}^{\bm\nu}u_\eps(\cdot;\bm y)\in V_\eps$.
    For every nonzero $\bm\nu\in\cF$, these derivatives are the
    unique solutions of the following variational problems: for every $v\in V_\eps$,
    \begin{equation}
        \label{eq:derivative-variational-equation}
        B_\eps(\partial_{\bm y}^{\bm\nu}u_\eps(\cdot;\bm y),v;\bm y)
        =-\sum_{\bm0<\bm\mu\leq\bm\nu}
        \binom{\bm\nu}{\bm\mu}
        \int_D
        \partial_{\bm y}^{\bm\mu}b(x;\bm y)
        \partial_{\bm y}^{\bm\nu-\bm\mu}u_\eps(x;\bm y)
        v(x)\,\dif x.
    \end{equation}
    Consequently,
    \begin{equation}
        \label{eq:derivative-recursion}
        \begin{split}
            \norm{\partial_{\bm y}^{\bm\nu}u_\eps(\cdot;\bm y)}_{1,\eps,D}
            \leq B
            \sum_{\bm0<\bm\mu\leq\bm\nu}
            \binom{\bm\nu}{\bm\mu}
            \norm{\partial_{\bm y}^{\bm\mu}b(\cdot;\bm y)}_{L^\infty(D)}
            \norm{\partial_{\bm y}^{\bm\nu-\bm\mu}u_\eps(\cdot;\bm y)}_{1,\eps,D},
        \end{split}
    \end{equation}
    where $B$ is the same $\eps$-independent stability constant as in
    \eqref{eq:B-constant}, \eqref{eq:solution-bound}, and
    \eqref{eq:coefficient-stability}.
\end{proposition}
\begin{proof}
    We prove the existence of the derivatives and
    \eqref{eq:derivative-variational-equation} simultaneously by induction on the
    order of the derivative. Throughout the proof, the perturbation parameter
    $\eps\in(0,1]$ is arbitrary but fixed.
    
    We begin with the first-order derivatives. Fix $j\geq1$ and $\bm y\in U$.
    For $h\neq0$ sufficiently small with $\bm y+h\bm e_j\in U$, set
    \begin{equation}
        \label{eq:first-difference-quotient}
        w_h(\cdot;\bm y)
        \coloneqq
        \delta_h^j u_\eps(\cdot;\bm y)
        \in V_\eps.
    \end{equation}
    Both the original and shifted parameters belong to $U$, so their coefficients
    have the same lower bound $b_0$. In particular, all stability estimates below
    use the constant $B$ from \eqref{eq:B-constant}.
    
    Subtracting the weak formulations at $\bm y+h\bm e_j$ and $\bm y$, and then
    dividing by $h$, shows that
    \begin{equation}
        \label{eq:first-difference-equation}
        B_\eps(w_h(\cdot;\bm y),v;\bm y)
        =-\int_D\delta_h^j b(x;\bm y)
        u_\eps(x;\bm y+h\bm e_j)v(x)\,\dif x,
        \qquad\forall v\in V_\eps.
    \end{equation}
    By \eqref{eq:coefficient-stability},
    $u_\eps(\cdot;\bm y+h\bm e_j)\to u_\eps(\cdot;\bm y)$ in $V_\eps$ as
    $h\to0$. Moreover,
    \begin{equation*}
        \delta_h^j b(\cdot;\bm y)
        \longrightarrow\partial_{y_j}b(\cdot;\bm y)
        \qquad\text{in }L^\infty(D).
    \end{equation*}
    Hence the right-hand side of \eqref{eq:first-difference-equation} converges in
    $V_\eps'$ to
    \begin{equation*}
        v\longmapsto-\int_D\partial_{y_j}b(x;\bm y)
        u_\eps(x;\bm y)v(x)\,\dif x.
    \end{equation*}
    Uniform coercivity therefore 
    implies that $w_h(\cdot;\bm y)$ converges in
    $V_\eps$ to the unique solution of
    \begin{equation}
        \label{eq:first-derivative-equation}
        B_\eps(\partial_{y_j}u_\eps(\cdot;\bm y),v;\bm y)
        =-\int_D\partial_{y_j}b(x;\bm y)
        u_\eps(x;\bm y)v(x)\,\dif x,
        \qquad\forall v\in V_\eps.
    \end{equation}
    This proves the assertion for all derivatives of order one. 
    To also verify continuity, we use the following standard comparison. 
    If
    $B_\eps(q(\bm y),v;\bm y)=F(\bm y)(v)$, 
    then
    \begin{equation}
        \label{eq:varying-problem-stability}
        \begin{split}
            \norm{q(\bm y)-q(\widetilde{\bm y})}_{1,\eps,D}
            \leq B\bigl(&\norm{F(\bm y)-F(\widetilde{\bm y})}_{-1,\eps,D}
            \\
            &+\norm{b(\cdot;\bm y)-b(\cdot;\widetilde{\bm y})}_{L^\infty(D)}
            \norm{q(\widetilde{\bm y})}_{1,\eps,D}\bigr).
        \end{split}
    \end{equation}
    Indeed, this follows by subtracting the two equations and using coercivity.
    The right-hand side in \eqref{eq:first-derivative-equation} is continuous in
    $V_\eps'$, so \eqref{eq:varying-problem-stability} shows that
    $\bm y\mapsto\partial_{y_j}u_\eps(\cdot;\bm y)$ is continuous as a
    $V_\eps$-valued map. The constant in this argument is independent of $\eps$.
    
    Assume now that, for some $n\geq1$, all mixed derivatives of order at most $n$
    exist, are continuous as $V_\eps$-valued functions, and satisfy
    \eqref{eq:derivative-variational-equation}. Let $\bm\nu\in\cF$ with
    $\abs{\bm\nu}=n$, fix $k\geq1$, and set
    \begin{equation}
        \label{eq:higher-difference-quotient}
        z_h(\cdot;\bm y)
        \coloneqq
        \delta_h^k\!\left(\partial_{\bm y}^{\bm\nu}u_\eps\right)
        (\cdot;\bm y).
    \end{equation}
    Subtracting the two equations for
    $\partial_{\bm y}^{\bm\nu}u_\eps(\cdot;\bm y+h\bm e_k)$ and
    $\partial_{\bm y}^{\bm\nu}u_\eps(\cdot;\bm y)$ yields, for every
    $v\in V_\eps$,
    \begin{equation}
        \label{eq:higher-difference-equation}
        \begin{aligned}
            B_{\eps}(z_{h}(\cdot;\bm{y}),v;\bm{y})
            &=
            -\int_{D} \delta_{h}^{k} b(x;\bm{y})
            \partial_{\bm{y}}^{\bm{\nu}}
            u_{\eps}(x;\bm{y}+h\bm{e}_{k})
            v(x)\,\dif{x}
            \\
            &\phantom{=}
            -\sum_{\bm{0}<\bm{\mu}\leq\bm{\nu}}
            \binom{\bm{\nu}}{\bm{\mu}}
            \int_{D} \delta_{h}^{k}\!\left[
            \partial_{\bm{y}}^{\bm{\mu}}b(\cdot;\bm{y})
            \right.
            \\
            &\qquad\left.
            {}\times
            \partial_{\bm{y}}^{\bm{\nu}-\bm{\mu}}
            u_{\eps}(\cdot;\bm{y})
            \right]\!(x)
            v(x)\,\dif{x}.
        \end{aligned}
    \end{equation}
    In the second line, $\delta_h^k$ acts on the complete
    parameter-dependent product in parentheses. 
    By the induction hypothesis and the regularity of
    $b(\cdot;\bm y)$, the right-hand side of
    \eqref{eq:higher-difference-equation} converges in $V_\eps'$ to
    \begin{align*}
        &-\int_D\partial_{y_k}b(x;\bm y)
        \partial_{\bm y}^{\bm\nu}u_\eps(x;\bm y)v(x)\,\dif x
        \\
        &-\sum_{\bm0<\bm\mu\leq\bm\nu}
        \binom{\bm\nu}{\bm\mu}
        \int_D\partial_{\bm y}^{\bm\mu+\bm e_k}b(x;\bm y)
        \partial_{\bm y}^{\bm\nu-\bm\mu}u_\eps(x;\bm y)v(x)\,\dif x
        \\
        &-\sum_{\bm0<\bm\mu\leq\bm\nu}
        \binom{\bm\nu}{\bm\mu}
        \int_D\partial_{\bm y}^{\bm\mu}b(x;\bm y)
        \partial_{\bm y}^{\bm\nu-\bm\mu+\bm e_k}u_\eps(x;\bm y)v(x)\,\dif x.
    \end{align*}
    All solution derivatives on the right-hand side have order at most $n$ and
    are therefore available by the induction hypothesis. Once again, uniform
    coercivity shows that $z_h(\cdot;\bm y)$ converges in $V_\eps$ to a unique
    limit. This limit is
    $\partial_{\bm y}^{\bm\nu+\bm e_k}u_\eps(\cdot;\bm y)$.
    
    Using the multi-index Pascal identity
    \begin{equation*}
        \binom{\bm\nu+\bm e_k}{\bm\kappa}
        =\binom{\bm\nu}{\bm\kappa}
        +\binom{\bm\nu}{\bm\kappa-\bm e_k},
    \end{equation*}
    with the convention that an inadmissible binomial coefficient is zero, the
    three terms in the preceding limit combine to
    \begin{equation*}
        \begin{split}
            &B_\eps(\partial_{\bm y}^{\bm\nu+\bm e_k}
            u_\eps(\cdot;\bm y),v;\bm y)
            \\
            &\qquad=-\sum_{\bm0<\bm\kappa\leq\bm\nu+\bm e_k}
            \binom{\bm\nu+\bm e_k}{\bm\kappa}
            \int_D\partial_{\bm y}^{\bm\kappa}b(x;\bm y)
            \partial_{\bm y}^{\bm\nu+\bm e_k-\bm\kappa}
            u_\eps(x;\bm y)v(x)\,\dif x.
        \end{split}
    \end{equation*}
    This is \eqref{eq:derivative-variational-equation} at order $n+1$. Its
    right-hand side depends continuously on $\bm y$, and the stability estimate
    \eqref{eq:varying-problem-stability} then also gives continuity of
    $\partial_{\bm y}^{\bm\nu+\bm e_k}u_\eps(\cdot;\bm y)$. The induction is
    complete. If the same multi-index is obtained by differentiating in a
    different order, both constructions solve \eqref{eq:derivative-variational-equation}
    with the same right-hand side. Uniqueness therefore shows that all such
    iterated derivatives agree.
    
    It remains to prove \eqref{eq:derivative-recursion}. In
    \eqref{eq:derivative-variational-equation}, choose
    $v=\partial_{\bm y}^{\bm\nu}u_\eps(\cdot;\bm y)$. By coercivity,
    Cauchy--Schwarz, and the inequality
    $\norm{v}_{L^2(D)}\leq\norm{v}_{1,\eps,D}$, we find
    \begin{align*}
        &\min\{1,b_0\}
        \norm{\partial_{\bm y}^{\bm\nu}u_\eps(\cdot;\bm y)}_{1,\eps,D}^2
        \\
        &\quad\leq
        \sum_{\bm0<\bm\mu\leq\bm\nu}
        \binom{\bm\nu}{\bm\mu}
        \norm{\partial_{\bm y}^{\bm\mu}b(\cdot;\bm y)}_{L^\infty(D)}
        \norm{\partial_{\bm y}^{\bm\nu-\bm\mu}u_\eps(\cdot;\bm y)}_{1,\eps,D}
        \norm{\partial_{\bm y}^{\bm\nu}u_\eps(\cdot;\bm y)}_{1,\eps,D}.
    \end{align*}
    If the final norm on the right-hand side is zero, the desired estimate is
    immediate; otherwise, we cancel it. Then \eqref{eq:derivative-recursion} follows with
    $B=1/\min\{1,b_0\}$. Thus the constant in the recursive estimate is precisely
    the same $\eps$-independent constant as in the basic solution estimate.
\end{proof}
The following elementary estimate controls the convolution in
\eqref{eq:derivative-recursion}.
\begin{lemma}[Gevrey convolution estimate]
    \label{lem:gevrey-convolution}
    Let $\mfs_{y}\geq1$, $\eta>1$, and let $\bm\nu\in\cF$ with $n\coloneqq\abs{\bm\nu}\geq1$. 
    Then
    \begin{equation}
        \label{eq:gevrey-convolution}
        \sum_{\bm0<\bm\mu\leq\bm\nu}
        \binom{\bm\nu}{\bm\mu}
        (\abs{\bm\mu}!)^{\mfs_{y}}
        (\abs{\bm\nu-\bm\mu}!)^{\mfs_{y}}
        \eta^{\abs{\bm\nu-\bm\mu}}
        \leq
        (n!)^{\mfs_{y}}\frac{\eta^n}{\eta-1}.
    \end{equation}
    The constant on the right-hand side is purely combinatorial and, in
    particular, independent of $\eps$.
\end{lemma}
\begin{proof}
    The generalized Vandermonde identity gives
    \begin{equation*}
        \sum_{\substack{\bm\mu\leq\bm\nu\\\abs{\bm\mu}=k}}
        \binom{\bm\nu}{\bm\mu}=\binom{n}{k}.
    \end{equation*}
    Moreover, for $1\leq k\leq n$ and $\mfs_{y}\geq1$,
    \begin{equation*}
        \binom{n}{k}(k!)^{\mfs_{y}}((n-k)!)^{\mfs_{y}}
        =(n!)^{\mfs_{y}}\binom{n}{k}^{1-\mfs_{y}}
        \leq(n!)^{\mfs_{y}}.
    \end{equation*}
    Grouping the left-hand side of \eqref{eq:gevrey-convolution} according to
    $k=\abs{\bm\mu}$ therefore yields
    \begin{equation*}
        \sum_{k=1}^n
        \binom{n}{k}(k!)^{\mfs_{y}}((n-k)!)^{\mfs_{y}}\eta^{n-k}
        \leq(n!)^{\mfs_{y}}\sum_{k=1}^n\eta^{n-k}
        \leq(n!)^{\mfs_{y}}\frac{\eta^n}{\eta-1}.
    \end{equation*}
\end{proof}
\subsection{\texorpdfstring{$\eps$}{ε}-robust \texorpdfstring{$\mfs_{y}$}{s}-Gevrey regularity of the solution}
\label{subsec:gevrey-solution}
\begin{theorem}[$\eps$-robust inheritance of Gevrey regularity]
    \label{thm:gevrey-regularity}
    Suppose that \cref{ass:uniform-bounds,ass:gevrey-b} hold. 
    Let $B$ be the $\eps$-independent stability constant in \eqref{eq:B-constant},
    \eqref{eq:solution-bound}, and \eqref{eq:derivative-recursion}, 
    and define
    \begin{equation}
        \label{eq:solution-weights}
        \eta\coloneqq1+BC_b,
        \qquad
        \widehat\beta_j\coloneqq\rho_b\eta\beta_j,
        \quad j\geq1.
    \end{equation}
    Then, for every $\eps\in(0,1]$ and $\bm\nu\in\cF$,
    \begin{equation}
        \label{eq:gevrey-solution-bound}
        \sup_{\bm y\in U}
        \norm{\partial_{\bm y}^{\bm\nu}u_\eps(\cdot;\bm y)}_{1,\eps,D}
        \leq
        B\norm{f}_{-1,\eps,D}
        (\abs{\bm\nu}!)^{\mfs_{y}}
        \widehat{\bm\beta}^{\bm\nu}
        \leq
        B\norm{f}_{L^2(D)}
        (\abs{\bm\nu}!)^{\mfs_{y}}
        \widehat{\bm\beta}^{\bm\nu}.
    \end{equation}
    Since
    \begin{equation*}
        \widehat{\bm\beta}^{\bm\nu}
        =
        (\rho_b\eta)^{\abs{\bm\nu}}\bm\beta^{\bm\nu},
    \end{equation*}
    \eqref{eq:gevrey-solution-bound} implies that, 
    for every $\eps\in(0,1]$,
    \begin{equation}
        \label{eq:solution-gevrey-membership}
        \bm y\longmapsto u_\eps(\cdot;\bm y)
        \in \cG_{\bm\beta}^{\mfs_y}(U;V_\eps).
    \end{equation}
    The corresponding Gevrey constants may be chosen as
    \begin{equation*}
        C_u=B\norm{f}_{L^2(D)},
        \qquad
        \rho_u=\rho_b\eta.
    \end{equation*}
    In particular, $C_u$, $\rho_u$, and $\bm\beta$ are independent of $\eps$.
\end{theorem}
\begin{proof}
    We prove the first inequality in \eqref{eq:gevrey-solution-bound} by induction
    on the order of the mixed derivative. The second inequality then follows from
    \eqref{eq:L2-dual-embedding}. Throughout the proof, every occurrence of $B$
    denotes the constant in \eqref{eq:B-constant}; in particular, no step of the
    induction introduces a constant depending on $\eps$.
    
    For $\bm\nu=\bm0$, the asserted estimate is exactly
    \eqref{eq:solution-bound}:
    \begin{equation*}
        \norm{u_\eps(\cdot;\bm y)}_{1,\eps,D}
        \leq B\norm{f}_{-1,\eps,D},
        \qquad\bm y\in U.
    \end{equation*}
    For a unit multi-index $\bm e_j$, the recursive estimate
    \eqref{eq:derivative-recursion}, the coefficient bound
    \eqref{eq:gevrey-b-bound}, and the zero-order estimate give
    \begin{align*}
        \norm{\partial_{y_j}u_\eps(\cdot;\bm y)}_{1,\eps,D}
        &\leq B
        \norm{\partial_{y_j}b(\cdot;\bm y)}_{L^\infty(D)}
        \norm{u_\eps(\cdot;\bm y)}_{1,\eps,D}
        \\
        &\leq B\norm{f}_{-1,\eps,D}(BC_b)\rho_b\beta_j
        \\
        &\leq B\norm{f}_{-1,\eps,D}\widehat\beta_j,
    \end{align*}
    where we used $BC_b<\eta$. This establishes the first-order case with the same
    $\eps$-independent constant $B$.
    
    Let now $n\geq1$ and assume that
    \begin{equation}
        \label{eq:induction-hypothesis}
        \norm{\partial_{\bm y}^{\bm\lambda}u_\eps(\cdot;\bm y)}_{1,\eps,D}
        \leq B\norm{f}_{-1,\eps,D}
        (\abs{\bm\lambda}!)^{\mfs_{y}}
        \widehat{\bm\beta}^{\bm\lambda}
    \end{equation}
    holds for every $\bm y\in U$ and every $\bm\lambda\in\cF$ with
    $\abs{\bm\lambda}\leq n$. Fix $\bm\nu\in\cF$ with
    $\abs{\bm\nu}=n$, a unit multi-index $\bm e_k$, and set
    \begin{equation*}
        \widetilde{\bm\nu}\coloneqq\bm\nu+\bm e_k,
        \qquad
        \abs{\widetilde{\bm\nu}}=n+1.
    \end{equation*}
    Applying \eqref{eq:derivative-recursion} to $\widetilde{\bm\nu}$ gives
    \begin{align*}
        &\norm{\partial_{\bm y}^{\widetilde{\bm\nu}}
            u_\eps(\cdot;\bm y)}_{1,\eps,D}
        \leq 
        B
        \sum_{\bm0<\bm\mu\leq\widetilde{\bm\nu}}
        \binom{\widetilde{\bm\nu}}{\bm\mu}
        \norm{\partial_{\bm y}^{\bm\mu}b(\cdot;\bm y)}_{L^\infty(D)}
        \norm{\partial_{\bm y}^{\widetilde{\bm\nu}-\bm\mu}
            u_\eps(\cdot;\bm y)}_{1,\eps,D}.
    \end{align*}
    Since $\bm\mu>\bm0$, the derivative of the solution in every summand has
    order at most $n$, so the induction hypothesis applies. Using
    \eqref{eq:gevrey-b-bound} and
    \begin{equation*}
        \widehat{\bm\beta}^{\widetilde{\bm\nu}-\bm\mu}
        =(\rho_b\eta)^{\abs{\widetilde{\bm\nu}-\bm\mu}}
        \bm\beta^{\widetilde{\bm\nu}-\bm\mu},
    \end{equation*}
    we obtain
    \begin{align*}
        &\norm{\partial_{\bm y}^{\widetilde{\bm\nu}}
            u_\eps(\cdot;\bm y)}_{1,\eps,D}\\
        &\quad\leq
        B\norm{f}_{-1,\eps,D}(BC_b)\rho_b^{n+1}
        \bm\beta^{\widetilde{\bm\nu}}\\
        &\qquad{}\times
        \sum_{\bm0<\bm\mu\leq\widetilde{\bm\nu}}
        \binom{\widetilde{\bm\nu}}{\bm\mu}
        (\abs{\bm\mu}!)^{\mfs_{y}}
        (\abs{\widetilde{\bm\nu}-\bm\mu}!)^{\mfs_{y}}
        \eta^{\abs{\widetilde{\bm\nu}-\bm\mu}}.
    \end{align*}
    The convolution estimate \eqref{eq:gevrey-convolution}, applied with
    $\abs{\widetilde{\bm\nu}}=n+1$, yields
    \begin{align*}
        &\norm{\partial_{\bm y}^{\widetilde{\bm\nu}}
            u_\eps(\cdot;\bm y)}_{1,\eps,D}
        \leq
        B\norm{f}_{-1,\eps,D}
        \frac{BC_b}{\eta-1}\rho_b^{n+1}
        \bm\beta^{\widetilde{\bm\nu}}
        ((n+1)!)^{\mfs_{y}}\eta^{n+1}.
    \end{align*}
    By the definition $\eta-1=BC_b$, the prefactor equals one. Therefore,
    \begin{align*}
        \norm{\partial_{\bm y}^{\widetilde{\bm\nu}}
            u_\eps(\cdot;\bm y)}_{1,\eps,D}
        &\leq
        B\norm{f}_{-1,\eps,D}
        ((n+1)!)^{\mfs_{y}}(\rho_b\eta)^{n+1}
        \bm\beta^{\widetilde{\bm\nu}}
        \\
        &=
        B\norm{f}_{-1,\eps,D}
        (\abs{\widetilde{\bm\nu}}!)^{\mfs_{y}}
        \widehat{\bm\beta}^{\widetilde{\bm\nu}}.
    \end{align*}
    This proves the induction step. Since the estimate is uniform in
    $\bm y\in U$, taking the supremum over $U$ gives the first inequality in
    \eqref{eq:gevrey-solution-bound}. The second follows from
    \eqref{eq:L2-dual-embedding}. All constants appearing in the proof are the
    $\eps$-independent constants from \eqref{eq:uniform-bounds},
    \eqref{eq:gevrey-b-bound}, and \eqref{eq:B-constant}.
\end{proof}
\begin{remark}[Uniformity with respect to $\eps$]
    \Cref{thm:gevrey-regularity} and, in particular, \cref{eq:solution-gevrey-membership} show that
    $u_\eps\in\cG_{\bm\beta}^{\mfs_y}(U;V_\eps)$ uniformly in $\eps$,
    in the sense that the same anisotropy sequence $\bm\beta$ and the same
    Gevrey constants $C_u$ and $\rho_u$ may be used for every $0 < \eps\leq 1$,
    provided $V_\eps$ be
    equipped with the norm $\norm{\cdot}_{1,\eps,D}$.
\end{remark}
\begin{remark}[Argument via the coefficient-to-solution map]
    The linearity of \eqref{eq:model-pde} in the unknown provides an alternative
    route to the parametric regularity result. 
    For fixed $f\in V_\eps'$ and for fixed $\eps\in(0,1]$, 
    consider the coefficient-to-solution map
    \begin{equation*}
        \mathcal S_\eps:c\longmapsto \mathcal S_\eps(c)\in V_\eps,
    \end{equation*}
    defined on the open subset of $L^\infty(D)$ consisting of uniformly positive
    reaction coefficients $c$. 
    Since the operator associated with
    \eqref{eq:bilinear-form} depends affinely on $c\in L^\infty(D)$, 
    its inverse at every $c=b(\cdot;\bm y)$ 
    admits an $\eps$-uniform bound $B$.
    A Neumann-series argument then shows
    that $\mathcal S_\eps$ is analytic on the uniform $L^{\infty}(D)$-neighborhood 
    \begin{equation*}
        \{b(\cdot;\bm y) + h : \bm{y} \in U,\, h \in L^\infty(D),\, B \norm{h}_{L^\infty(D)} < 1/2\},
    \end{equation*}
    where the inverses are bounded by $2B$, independently of $\eps$. 
    For fixed $f\in L^2(D)$, the
    parameter-to-solution map may therefore be written as the composition
    \begin{equation*}
        \bm y\longmapsto u_\eps(\cdot;\bm y)
        =\mathcal S_\eps\bigl(b(\cdot;\bm y)\bigr).
    \end{equation*}
    On each finite-dimensional parameter section, the composition of an analytic
    map with an $\mfs_{y}$-Gevrey map is again $\mfs_{y}$-Gevrey, with the same Gevrey
    index $\mfs_{y}$. This yields an alternative proof of the qualitative regularity assertion in
    \cref{thm:gevrey-regularity}. The direct inductive proof above is retained
    because it provides the explicit, quantitative
    anisotropic derivative bounds in \eqref{eq:gevrey-solution-bound}.
\end{remark}
\begin{corollary}[Preservation of anisotropic summability]
    \label{cor:summability}
    If $\bm\beta\in\ell^p(\bN)$ for some $p>0$, then
    $\widehat{\bm\beta}\in\ell^p(\bN)$ and
    \begin{equation*}
        \norm{\widehat{\bm\beta}}_{\ell^p(\bN)}
        =\rho_b\eta\norm{\bm\beta}_{\ell^p(\bN)}.
    \end{equation*}
    The factor $\rho_b\eta$ is independent of $\eps$. Thus, the passage from the
    coefficient to the solution does not change the summability exponent relevant
    for dimension truncation and QMC integration.
\end{corollary}
\begin{remark}[Consistency with the affine case]
    If
    \begin{equation*}
        b(x;\bm y)=\overline b(x)+\sum_{j\geq1}y_j\psi_j(x),
    \end{equation*}
    where
    $\sum_{j\geq1}\norm{\psi_j}_{L^\infty(D)}<\infty$, then
    $\partial_{\bm y}^{\bm\mu}b(\cdot;\bm y)=0$ whenever
    $\abs{\bm\mu}\geq2$. Equation
    \eqref{eq:derivative-variational-equation} then reduces to the one-step
    recursion from the affine-parametric analysis and gives the sharper estimate
    \begin{equation*}
        \sup_{\bm y\in U}
        \norm{\partial_{\bm y}^{\bm\nu}u_\eps(\cdot;\bm y)}_{1,\eps,D}
        \leq B\norm{f}_{-1,\eps,D}\abs{\bm\nu}!
        \prod_{j\in\supp(\bm\nu)}
        \bigl(B\norm{\psi_j}_{L^\infty(D)}\bigr)^{\nu_j}.
    \end{equation*}
    All constants are independent of $\eps$.
\end{remark}
\section{Dimension truncation}
\label{sec:dim-trunc}
Our model for the parametric coefficient $b$ involves countably many
parameters. Although the solution possesses the required parametric
regularity, practical computations require truncating the parameter dependence
to a finite dimension $s$. In this section, we therefore establish a general
result concerning the dimension truncation of smooth functions depending on
countably many parameters.

Dimension truncation is frequently studied in the context of QMC for
PDEs with uncertain input data from function spaces. 
A setting closely related to ours was considered in \cite{kuo2012QMC}. For
affine-parametric operator equations, an improved convergence rate was later
established in \cite{gantner2018dimension} by means of a Neumann series
argument. Since that approach relies on the affine structure, we instead use
the Taylor expansion argument from
\cite{guth2022application,guth2024lognormal}.
On $U = [-1/2,1/2]^{\bN}$, denote by
\begin{equation*}
    \bm\pi(\dif\bm y)\coloneqq\bigotimes_{j\geq1}\dif y_j
\end{equation*}
the uniform product probability measure on $U$. Let $X$ be a separable Banach
space and let $g:U\to X$. For
$\bm y=(y_j)_{j\geq1}\in U$ and $s\in\bN$, write
\begin{equation*}
    \bm y_{\leq s}\coloneqq(y_1,\ldots,y_s),
    \qquad
    \bm y_{>s}\coloneqq(y_{s+1},y_{s+2},\ldots).
\end{equation*}
For a tail sequence $\bm z_{>s}=(z_{s+1},z_{s+2},\ldots)$, we use the
concatenation notation
\begin{equation*}
    (\bm y_{\leq s},\bm z_{>s})
    \coloneqq(y_1,\ldots,y_s,z_{s+1},z_{s+2},\ldots).
\end{equation*}
In particular,
\begin{equation*}
    (\bm y_{\leq s},\bm0)=(y_1,\ldots,y_s,0,0,\ldots),
    \qquad
    g_s(\bm y)\coloneqq g(\bm y_{\leq s},\bm0).
\end{equation*}

We work under the following assumptions:
\begin{enumerate}[label=\textbf{D.\arabic*}]
    \item
    \label{item:anchored-convergence}
    For $\bm\pi$-a.e. $\bm y\in U$,
    \begin{equation*}
        \norm{g(\bm y)-g_s(\bm y)}_X\longrightarrow0
        \qquad\text{as }s\to\infty.
    \end{equation*}
    
    \item
    \label{item:gevrey-truncation-regularity}
    There exist $0<p<1$ and a nonnegative, non-increasing sequence
    $\bm\beta=(\beta_j)_{j\geq1}\in\ell^p(\bN)$ such that
    \begin{equation*}
        g\in\cG_{\bm\beta}^{\mfs_y}(U;X).
    \end{equation*}
    We fix constants $C_g,\rho_g>0$ for which
    \begin{equation}
        \label{eq:abstract-gevrey-truncation-bound}
        \sup_{\bm y\in U}
        \norm{\partial_{\bm y}^{\bm\nu}g(\bm y)}_X
        \leq C_g(\abs{\bm\nu}!)^{\mfs_{y}}
        \rho_g^{\abs{\bm\nu}}\bm\beta^{\bm\nu}
        \qquad\forall\bm\nu\in\cF.
    \end{equation}
\end{enumerate}

\begin{proposition}[Dimension truncation for $\mfs_{y}$-Gevrey maps]
    \label{prop:abstract-dimension-truncation}
    Suppose that \cref{item:anchored-convergence,item:gevrey-truncation-regularity}
    hold. Then
    \begin{equation}
        \label{eq:abstract-dimension-truncation-bound}
        \norm*{\int_U\bigl(g(\bm y)-g_s(\bm y)\bigr)
            \,\bm\pi(\dif\bm y)}_X
        \leq C_{\mathrm{DT}}s^{-2/p+1},
    \end{equation}
    where $C_{\mathrm{DT}}>0$ is independent of $s$. 
    Moreover, for every $G\in X'$,
    \begin{equation}
        \label{eq:abstract-functional-dimension-truncation-bound}
        \abs*{\int_U G\bigl(g(\bm y)-g_s(\bm y)\bigr)
            \,\bm\pi(\dif\bm y)}
        \leq C_{\mathrm{DT}}\norm{G}_{X'}s^{-2/p+1}.
    \end{equation}
    The constant depends on $C_g$, $\rho_g$, $p$, $\bm\beta$, and the Gevrey index
    $\mfs_{y}$, but not on the truncation dimension $s$.
\end{proposition}

\begin{proof}
    We adapt the Taylor argument in
    \cite{guth2022application,guth2024lognormal}. Set
    \begin{equation*}
        k\coloneqq\ceil{\frac{1}{1-p}}
    \end{equation*}
    and fix $G\in X'$. For
    \begin{equation*}
        F_G(\bm y)\coloneqq G(g(\bm y)),
    \end{equation*}
    \cref{item:gevrey-truncation-regularity} gives
    \begin{equation}
        \label{eq:scalar-gevrey-bound}
        \abs{\partial_{\bm y}^{\bm\nu}F_G(\bm y)}
        \leq C_g\norm{G}_{X'}(\abs{\bm\nu}!)^{\mfs_{y}}
        \rho_g^{\abs{\bm\nu}}\bm\beta^{\bm\nu},
        \qquad \bm\nu\in\cF,
    \end{equation}
    uniformly in $\bm y\in U$. Taylor expansion about
    $(\bm y_{\leq s},\bm0)$, with integral remainder, yields
    \begin{equation}
        \label{eq:scalar-tail-taylor}
        \begin{split}
            F_G(\bm y)-F_G(\bm y_{\leq s},\bm0)
            ={}&\sum_{\ell=1}^k
            \sum_{\substack{\abs{\bm\nu}=\ell\\
                    \nu_j=0,\ j\leq s}}
            \frac{\bm y^{\bm\nu}}{\bm\nu!}
            \partial_{\bm y}^{\bm\nu}F_G(\bm y_{\leq s},\bm0)
            \\
            &+(k+1)
            \sum_{\substack{\abs{\bm\nu}=k+1\\
                    \nu_j=0,\ j\leq s}}
            \frac{\bm y^{\bm\nu}}{\bm\nu!}
            \int_0^1(1-t)^k
            \partial_{\bm y}^{\bm\nu}F_G
            (\bm y_{\leq s},t\bm y_{>s})\,\dif t.
        \end{split}
    \end{equation}
    After integration over $U$, denote the Taylor polynomial and the integral
    remainder in \eqref{eq:scalar-tail-taylor} by $I_s$ and $II_s$,
    respectively.
    
    The derivative in every summand of $I_s$ is independent of the tail
    variables. Since the uniform measure is centered, Fubini's theorem shows that
    a summand vanishes whenever $\nu_j=1$ for some $j>s$. In particular, all
    first-order terms vanish. 
    Using \eqref{eq:scalar-gevrey-bound} and
    $\abs{y_j}\leq1/2$, we obtain 
    \begin{equation}
        \label{eq:I-s-first-bound}
        \abs{I_s}
        \leq{} C_g \left(\max_{2 \le \ell \le k} (\ell!)^{\mfs_{y}} (\rho_g/2)^{\ell}\right)\norm{G}_{X'}
        \sum_{\ell=2}^k
        \sum_{\substack{\abs{\bm\nu}=\ell\\
                \nu_j=0,\ j\leq s\\
                \nu_j\neq1,\ j>s}}
        \frac{\bm\beta^{\bm\nu}}{\bm\nu!}.
    \end{equation}
    The remaining sum satisfies
    \begin{equation*}
        \sum_{\ell=2}^k
        \sum_{\substack{\abs{\bm\nu}=\ell\\
                \nu_j=0,\ j\leq s\\
                \nu_j\neq1,\ j>s}}
        \frac{\bm\beta^{\bm\nu}}{\bm\nu!}
        \leq
        -1+\prod_{j>s}\left(1+\sum_{r=2}^k\frac{\beta_j^r}{r!}\right)
        \leq
        -1+\exp\left(e^{\beta_1}\sum_{j>s}\beta_j^2\right).
    \end{equation*}
    Here we used the monotonicity of $\bm \beta$,
    \begin{equation*}
        \sum_{r=2}^k\frac{\beta_j^r}{r!}
        \leq \beta_j^2\sum_{r=0}^{k-2}\frac{\beta_1^r}{r!}
        \leq e^{\beta_1}\beta_j^2,
    \end{equation*}
    and $\log(1+x)\leq x$. By Stechkin's inequality,
    \begin{equation*}
        \sum_{j>s}\beta_j^2
        \leq\norm{\bm\beta}_{\ell^p(\bN)}^2s^{-2/p+1}.
    \end{equation*}
    Since $s^{-2/p+1}\in(0,1]$, convexity of the exponential gives 
    \begin{equation}
        \label{eq:I-s-final-bound}
        \abs{I_s}\leq C_I\norm{G}_{X'}s^{-2/p+1},
    \end{equation}
    where $C_I$ is independent of $s$ and $G$.
    
    For $II_s$, no centeredness argument is available because the derivative in
    the remainder depends on the tail variables. Since
    $\int_0^1(1-t)^k\,\dif t=1/(k+1)$, the prefactor $k+1$ cancels, and
    \eqref{eq:scalar-gevrey-bound} yields
    \begin{align*}
        \abs{II_s}
        &\leq C_g\norm{G}_{X'}((k+1)!)^{\mfs_{y}}
        \left(\frac{\rho_g}{2}\right)^{k+1}
        \sum_{\substack{\abs{\bm\nu}=k+1\\
                \nu_j=0,\ j\leq s}}
        \frac{\bm\beta^{\bm\nu}}{\bm\nu!}\\
        &= C_g\norm{G}_{X'}((k+1)!)^{\mfs_{y}-1}
        \left(\frac{\rho_g}{2}\right)^{k+1}
        \left(\sum_{j>s}\beta_j\right)^{k+1}.
    \end{align*}
    In the last step, we used the exact multinomial identity
    \begin{equation*}
        \sum_{\substack{\abs{\bm\nu}=k+1\\
                \nu_j=0,\ j\leq s}}
        \frac{\bm\beta^{\bm\nu}}{\bm\nu!}
        =\frac{1}{(k+1)!}\left(\sum_{j>s}\beta_j\right)^{k+1}.
    \end{equation*}
    Another application of Stechkin's inequality gives
    \begin{equation}
        \label{eq:II-s-final-bound}
        \abs{II_s}
        \leq C_{II}\norm{G}_{X'}
        s^{(k+1)(-1/p+1)},
    \end{equation}
    where $C_{II}$ is independent of $s$ and $G$. The choice of $k$ ensures that
    \begin{equation*}
        (k+1)\left(\frac1p-1\right)\geq\frac2p-1.
    \end{equation*}
    Combining \eqref{eq:I-s-final-bound} and \eqref{eq:II-s-final-bound} therefore
    proves \eqref{eq:abstract-functional-dimension-truncation-bound}. Taking the
    supremum over all $G\in X'$ with $\norm{G}_{X'}\leq1$ proves
    \eqref{eq:abstract-dimension-truncation-bound}.
\end{proof}

\begin{remark}
    The proof shows that $C_{\mathrm{DT}}$ depends on $g$ only through the
    derivative-bound constants in \cref{item:gevrey-truncation-regularity} and on
    the sequence $\bm\beta$. Thus, if $g=g_\eps$ and these quantities are independent
    of $\eps$, then so is $C_{\mathrm{DT}}$.
\end{remark}

We now apply \cref{prop:abstract-dimension-truncation} to
$g(\bm y)=u_\eps(\cdot;\bm y)$ with $X=V_\eps$ and norm
$\norm{\cdot}_{1,\eps,D}$. 

\begin{assumption}[Summability and ordering for dimension truncation]
    \label{ass:dimension-truncation-weights}
    For some $0<p<1$, the sequence from \cref{ass:gevrey-b} satisfies
    \begin{equation*}
        \bm\beta\in\ell^p(\bN),
        \qquad
        \beta_1\geq\beta_2\geq\cdots\geq0.
    \end{equation*}
\end{assumption}

For $s\in\bN$, define
\begin{equation}
    \label{eq:dimension-truncated-solution}
    u_{\eps,s}(\cdot;\bm y)
    \coloneqq u_\eps(\cdot;(\bm y_{\leq s},\bm0)).
\end{equation}
Equivalently, $u_{\eps,s}$ solves \eqref{eq:weak-problem} with reaction
coefficient $b(\cdot;(\bm y_{\leq s},\bm0))$.

\begin{corollary}[$\eps$-robust dimension truncation]
    \label{cor:dimension-truncation}
    Suppose that
    \cref{ass:uniform-bounds,ass:gevrey-b,ass:dimension-truncation-weights} hold
    and let $f\in L^2(D)$. Then, for every $\eps\in(0,1]$ and $s\in\bN$,
    \begin{equation}
        \label{eq:solution-dimension-truncation}
        \norm*{\int_U\bigl(u_\eps(\cdot;\bm y)
            -u_{\eps,s}(\cdot;\bm y)\bigr)\,\bm\pi(\dif\bm y)}_{1,\eps,D}
        \leq C_{\mathrm{DT}}s^{-2/p+1},
    \end{equation}
    where $C_{\mathrm{DT}}>0$ is independent of $s$ and $\eps$. Moreover, for
    every $G\in(V_\eps)'$,
    \begin{equation}
        \label{eq:functional-dimension-truncation}
        \abs*{\int_U G\bigl(u_\eps(\cdot;\bm y)
            -u_{\eps,s}(\cdot;\bm y)\bigr)\,\bm\pi(\dif\bm y)}
        \leq C_{\mathrm{DT}}\norm{G}_{(V_\eps)'}s^{-2/p+1},
    \end{equation}
    with the same constant as in \eqref{eq:solution-dimension-truncation}.
\end{corollary}

\begin{proof}
    For each fixed $\eps\in(0,1]$, \cref{prop:derivative-equation} gives the
    required smoothness of the map
    $\bm y\mapsto u_\eps(\cdot;\bm y)$. We first verify
    \cref{item:anchored-convergence}. For $m>s$, the fundamental theorem of
    calculus on the segment from $(\bm y_{\leq s},\bm0)$ to
    $(\bm y_{\leq m},\bm0)$ and \eqref{eq:gevrey-b-bound} at first order give
    \begin{equation*}
        \norm{b(\cdot;(\bm y_{\leq m},\bm0))
            -b(\cdot;(\bm y_{\leq s},\bm0))}_{L^\infty(D)}
        \leq\frac{C_b\rho_b}{2}\sum_{j=s+1}^m\beta_j.
    \end{equation*}
    Since $p<1$, $\bm\beta\in\ell^1(\bN)$. Moreover,
    $(\bm y_{\leq m},\bm0)\to\bm y$ in the product topology, so the continuity in
    \cref{ass:gevrey-b} permits passage to the limit $m\to\infty$. Hence
    \eqref{eq:coefficient-stability} and \eqref{eq:L2-dual-embedding} imply
    \begin{equation*}
        \sup_{\bm y\in U}
        \norm{u_\eps(\cdot;\bm y)-u_{\eps,s}(\cdot;\bm y)}_{1,\eps,D}
        \leq
        \frac{B^2C_b\rho_b}{2}\norm{f}_{L^2(D)}
        \sum_{j>s}\beta_j\longrightarrow0.
    \end{equation*}
    Thus \cref{item:anchored-convergence} holds, uniformly in $\eps$.
    
    Next, \cref{thm:gevrey-regularity} gives, for every $\bm\nu\in\cF$,
    \begin{equation*}
        \sup_{\bm y\in U}
        \norm{\partial_{\bm y}^{\bm\nu}u_\eps(\cdot;\bm y)}_{1,\eps,D}
        \leq B\norm{f}_{L^2(D)}
        (\abs{\bm\nu}!)^{\mfs_{y}}
        (\rho_b\eta)^{\abs{\bm\nu}}\bm\beta^{\bm\nu}.
    \end{equation*}
    Therefore \cref{item:gevrey-truncation-regularity} holds with
    \begin{equation*}
        C_g=B\norm{f}_{L^2(D)},
        \qquad \rho_g=\rho_b\eta,
        \qquad \bm\beta=(\beta_j)_{j\geq1}.
    \end{equation*}
    These quantities are independent of $\eps$. The assertions now follow from
    \cref{prop:abstract-dimension-truncation}.
\end{proof}

\section{Higher-order Quasi--Monte Carlo}
\label{sec:qmc}

This section recalls the higher-order quasi--Monte Carlo (QMC) result used for
the dimension-truncated parametric problem. We consider the numerical
approximation of integrals of functions $F$ defined on the $s$-dimensional unit
cube $[0,1]^s$,
\begin{equation*}
    I_s(F)\coloneqq\int_{[0,1]^s}F(\bm y)\,\dif\bm y,
\end{equation*}
by QMC quadrature rules
\begin{equation*}
    Q_{N,s}(F)\coloneqq\frac1N\sum_{n=0}^{N-1}F(\bm y_n),
\end{equation*}
where the quadrature points $\bm y_0,\ldots,\bm y_{N-1}\in[0,1]^s$ are
chosen according to a deterministic lattice rule belonging to the family of
digital nets. In the application to the parametric reaction--diffusion
problem, these points are translated from $[0,1]^s$ to
$[-1/2,1/2]^s$.

\subsection{Weighted spaces and SPOD weights}

We use the higher-order unanchored Sobolev spaces from
\cite{dick2014higher} in the case $r=\infty$, which is the only case needed
below.\footnote{We follow the definition in \cite{dick2016multilevel}, which
    corrects a typographical error in the corresponding definition of
    \cite{dick2014higher}.} For $k,\ell \in \bN_{0}$ with $k \le \ell$, set
$\{k:\ell\}\coloneqq\{k,k+1,\ldots,\ell\}$. If
$\mfu\subseteq\{1:s\}$, $\mfv\subseteq\mfu$, and
$\bm\tau_{\mfu\setminus\mfv}\in
\{1:\alpha\}^{\abs{\mfu\setminus\mfv}}$, then
\begin{equation*}
    (\bm\alpha_{\mfv},\bm\tau_{\mfu\setminus\mfv},\bm0)
\end{equation*}
denotes the multi-index $\bm\nu\in\{0,\ldots,\alpha\}^s$ with
$\nu_j=\alpha$ for $j\in\mfv$, $\nu_j=\tau_j$ for
$j\in\mfu\setminus\mfv$, and $\nu_j=0$ for $j\notin\mfu$.

\begin{definition}[Higher-order unanchored Sobolev norm]
    Let $\alpha,s\in\bN$, $1\leq q\leq\infty$, and
    $1\leq r\leq\infty$, and let
    $\bgg=(\gamma_{\mfu})_{\mfu\subset\bN}$ be a collection of
    nonnegative real weights indexed by the finite subsets of $\bN$. 
    Suppose
    that $F:[0,1]^s\to\bR$ has partial derivatives of orders up to $\alpha$
    with respect to each variable. Define
    \begin{align*}
        \norm{F}_{s,\alpha,\bgg,q,r}
        \coloneqq{}&
        \left[
        \sum_{\mfu\subseteq\{1:s\}}
        \left(
        \frac{1}{\gamma_{\mfu}^{q}}
        \sum_{\mfv\subseteq\mfu}
        \sum_{\bm\tau_{\mfu\setminus\mfv}
            \in\{1:\alpha\}^{\abs{\mfu\setminus\mfv}}}
        \right.
        \right.\\
        &\left.
        \left.
        \int_{[0,1]^{\abs{\mfv}}}
        \abs*{
            \int_{[0,1]^{s-\abs{\mfv}}}
            \left(
            \partial_{\bm y}^{
                (\bm\alpha_{\mfv},
                \bm\tau_{\mfu\setminus\mfv},\bm0)}F
            \right)(\bm y)
            \dif\bm y_{\{1:s\}\setminus\mfv}
        }^{q}
        \dif\bm y_{\mfv}
        \right)^{r/q}
        \right]^{1/r},
    \end{align*}
    with the usual modifications if $q$ or $r$ is infinite. Let
    $\mathcal W_{s,\alpha,\bgg,q,r}$ denote the Banach space of all such
    functions with finite norm. We exclusively use $r=\infty$ below.
\end{definition}

We use this norm for integrands satisfying the Gevrey derivative bound
\begin{equation}
    \label{eq:ass-QMC-integrand}
    \abs{\partial_{\bm y}^{\bm\nu}F(\bm y)}
    \leq
    C_1(\abs{\bm\nu}!)^{\mfs_{y}}
    C_2^{\abs{\bm\nu}}\bm\beta^{\bm\nu},
    \qquad
    \bm\nu\in\{0,1,\ldots,\alpha\}^s,
    \quad \bm y\in[0,1]^s,
\end{equation}
where $C_1,C_2>0$, $\mfs_{y}\geq1$, and
$\bm\beta=(\beta_j)_{j\geq1}\in\ell^p(\bN)$ for some
$0<p\leq1/\mfs_{y}$. Coordinates with $\beta_j=0$ are inactive by
\eqref{eq:ass-QMC-integrand} and may be omitted, so we assume without loss of
generality that the active weights are positive.

\subsubsection*{SPOD weights}

We now choose weights $\bgg$ such that an integrand satisfying
\eqref{eq:ass-QMC-integrand} belongs to
$\mathcal W_{s,\alpha,\bgg,q,\infty}$ with a norm bound independent of
$s$. For $1\leq q<\infty$, the definition of the norm, the fact that all
integration domains have measure one, and \eqref{eq:ass-QMC-integrand} give
\begin{align*}
    \norm{F}_{s,\alpha,\bgg,q,\infty}
    &\leq
    C_1\max_{\mfu\subseteq\{1:s\}}\gamma_{\mfu}^{-1}
    \left[
    \sum_{\mfv\subseteq\mfu}
    \sum_{\bm\tau_{\mfu\setminus\mfv}
        \in\{1:\alpha\}^{\abs{\mfu\setminus\mfv}}} \left(
    \abs*{(\bm\alpha_{\mfv},
        \bm\tau_{\mfu\setminus\mfv},\bm0)}!
    \right)^{\mfs_{y} q}
    \right.\\
    &\hspace{3.5cm}\left.
    {}\times
    C_2^{q\abs*{(\bm\alpha_{\mfv},
            \bm\tau_{\mfu\setminus\mfv},\bm0)}}
    \left(
    \bm\beta^{(\bm\alpha_{\mfv},
        \bm\tau_{\mfu\setminus\mfv},\bm0)}
    \right)^q
    \right]^{1/q}\\
    &\leq
    C_1\max_{\mfu\subseteq\{1:s\}}\gamma_{\mfu}^{-1}
    \sum_{\mfv\subseteq\mfu}
    \sum_{\bm\tau_{\mfu\setminus\mfv}
        \in\{1:\alpha\}^{\abs{\mfu\setminus\mfv}}}
    \left(
    \abs*{(\bm\alpha_{\mfv},
        \bm\tau_{\mfu\setminus\mfv},\bm0)}!
    \right)^{\mfs_{y}}\\
    &\hspace{3.5cm}\times
    C_2^{\abs*{(\bm\alpha_{\mfv},
            \bm\tau_{\mfu\setminus\mfv},\bm0)}}
    \bm\beta^{(\bm\alpha_{\mfv},
        \bm\tau_{\mfu\setminus\mfv},\bm0)}\\
    &=
    C_1\max_{\mfu\subseteq\{1:s\}}\gamma_{\mfu}^{-1}
    \sum_{\bm\nu_{\mfu}\in\{1:\alpha\}^{\abs{\mfu}}}
    C_2^{\abs{\bm\nu_{\mfu}}}
    (\abs{\bm\nu_{\mfu}}!)^{\mfs_{y}}
    \prod_{j\in\mfu}
    \left(2^{\delta(\nu_j,\alpha)}\beta_j^{\nu_j}\right).
\end{align*}
Here $\delta$ denotes the Kronecker delta. The factor
$2^{\delta(\nu_j,\alpha)}$ occurs because a derivative of order $\alpha$ in
coordinate $j$ can arise either from $j\in\mfv$ or from
$\tau_j=\alpha$ with $j\in\mfu\setminus\mfv$. For $q=\infty$, the usual
modification of the norm gives the same final bound.

We therefore define the Gevrey-SPOD weights by $\gamma_{\emptyset}=1$ and
\begin{equation}
    \label{eq:SPOD-weights}
    \gamma_{\mfu}
    \coloneqq
    \sum_{\bm\nu_{\mfu}\in\{1:\alpha\}^{\abs{\mfu}}}
    C_2^{\abs{\bm\nu_{\mfu}}}
    (\abs{\bm\nu_{\mfu}}!)^{\mfs_{y}}
    \prod_{j\in\mfu}
    \left(2^{\delta(\nu_j,\alpha)}\beta_j^{\nu_j}\right),
\end{equation}
for every nonempty finite $\mfu\subset\bN$. With this choice,
\begin{equation}
    \label{eq:gevrey-spod-norm-bound}
    \norm{F}_{s,\alpha,\bgg,q,\infty}\leq C_1,
    \qquad 1\leq q\leq\infty.
\end{equation}

\subsection{Interlaced polynomial lattice rules}

We recall the higher-order QMC rules used below. Interlaced polynomial lattice
rules were introduced in \cite{Dick2007Explicit,Dick2008Walsh} and belong to
the class of higher-order digital nets; see also \cite{dick2010digital}. We use
the notation from \cite{dick2014higher}.

Let $\mathfrak b$ be prime, let $\bZ_{\mathfrak b}$ be the finite field with
$\mathfrak b$ elements, and let $\bZ_{\mathfrak b}[x]$ denote the polynomials
over $\bZ_{\mathfrak b}$.

\begin{definition}[Polynomial lattice point set]
    For a prime $\mathfrak b$ and $m\in\bN$, let
    $P\in\bZ_{\mathfrak b}[x]$ be an irreducible polynomial with
    $\deg(P)=m$. For a dimension $s\geq1$, select
    $q_1(x),\ldots,q_s(x)$ from
    \begin{equation*}
        \cG_{\mathfrak b,m}
        \coloneqq
        \{q(x)\in\bZ_{\mathfrak b}[x]\setminus\{0\}:\deg(q)<m\},
    \end{equation*}
    and write
    \begin{equation*}
        \bm q=\bm q(x)=(q_1(x),\ldots,q_s(x))
        \in\cG_{\mathfrak b,m}^s.
    \end{equation*}
    For $0\leq n<\mathfrak b^m$, let
    $n=\eta_0+\eta_1\mathfrak b+\cdots+\eta_{m-1}\mathfrak b^{m-1}$
    be its $\mathfrak b$-adic expansion and associate with it the polynomial
    \begin{equation*}
        n(x)=\sum_{r=0}^{m-1}\eta_rx^r\in\bZ_{\mathfrak b}[x].
    \end{equation*}
    For a formal Laurent series, define
    \begin{equation*}
        v_m\left(\sum_{\ell=w}^{\infty}t_\ell x^{-\ell}\right)
        \coloneqq
        \sum_{\ell=\max\{1,w\}}^m t_\ell\mathfrak b^{-\ell}.
    \end{equation*}
    The polynomial lattice point set
    $\bm{\cS}_{P,\mathfrak b,m,s}(\bm q)$ consists of
    \begin{equation*}
        \bm y_n=
        \left(
        v_m\left(\frac{n(x)q_1(x)}{P(x)}\right),\ldots,
        v_m\left(\frac{n(x)q_s(x)}{P(x)}\right)
        \right)\in[0,1)^s,
        \quad n=0,\ldots,\mathfrak b^m-1.
    \end{equation*}
\end{definition}

\begin{definition}[Digit interlacing]
    For an interlacing factor $\alpha\in\bN$, define
    $\cD_\alpha:[0,1)^\alpha\to[0,1)$ as follows. If
    $x_r=\sum_{a=1}^{\infty}\xi_{r,a}\mathfrak b^{-a}$ for
    $r=1,\ldots,\alpha$, then
    \begin{equation*}
        \cD_\alpha(x_1,\ldots,x_\alpha)
        \coloneqq
        \sum_{a=1}^{\infty}\sum_{r=1}^{\alpha}
        \xi_{r,a}\mathfrak b^{-r-(a-1)\alpha}.
    \end{equation*}
    For vectors in $[0,1)^{\alpha s}$, apply this map componentwise in
    consecutive blocks of length $\alpha$. An interlaced polynomial lattice
    rule of order $\alpha$ in $s$ dimensions uses the points
    \begin{equation*}
        \cD_\alpha\bigl(\bm{\cS}_{P,\mathfrak b,m,\alpha s}(\bm q)\bigr),
        \qquad \bm q\in\cG_{\mathfrak b,m}^{\alpha s}.
    \end{equation*}
\end{definition}

For functions in $\mathcal W_{s,\alpha,\bgg,q,\infty}$, the following CBC
construction and error bound are known.

\begin{theorem}[{\cite[Theorem~3.10]{dick2014higher}}]
    \label{thm:ho-qmc-error}
    Let $\alpha,s\in\bN$ with $\alpha>1$, let
    $1\leq q\leq\infty$, and let
    $\bgg=(\gamma_{\mfu})_{\mfu\subset\bN}$ be a collection of weights.
    Let $\mathfrak b$ be prime and $m\in\bN$. Then an interlaced polynomial
    lattice rule of order $\alpha$ with $N=\mathfrak b^m$ points can be
    constructed using a component-by-component algorithm such that, for every
    $\lambda\in(1/\alpha,1]$,
    \begin{equation}
        \label{eq:cbc-bound-estimate}
        \abs{Q_{N,s}(F)-I_s(F)}
        \leq
        \left(
        \frac{2}{N-1}
        \sum_{\emptyset\neq\mfu\subseteq\{1:s\}}
        \gamma_{\mfu}^{\lambda}
        [\rho_{\alpha,\mathfrak b}(\lambda)]^{\abs{\mfu}}
        \right)^{1/\lambda}
        \norm{F}_{s,\alpha,\bgg,q,\infty},
    \end{equation}
    where $\rho_{\alpha,\mathfrak b}(\lambda)$ is the constant from
    \cite[Theorem~3.10]{dick2014higher}.
\end{theorem}
\begin{remark}[CBC construction]
    \label{rem:cost-cbc}
    Using the criterion $E_d(\bm q)$ from
    \cite[Eq.~(3.31)]{dick2014higher} with our Gevrey--SPOD weights
    \eqref{eq:SPOD-weights}, an interlaced polynomial lattice rule of order
    $\alpha$ with $N=\mathfrak b^m$ points can be constructed such that the
    error bound \eqref{eq:cbc-bound-estimate} holds, with total computational
    cost
    \begin{equation*}
        \cO\bigl(\alpha s N \log N+\alpha^{2} s^{2} N \bigr),
    \end{equation*}
    see \cite[Section~3.4]{dick2014higher}. The fast CBC algorithm developed
    there for the analytic SPOD weights must be adapted to account for the
    additional Gevrey exponent $\mfs_{y}$ in \eqref{eq:SPOD-weights}. This
    amounts to replacing the factorial ratios in the recursive updates by
    their $\mfs_{y}$-th powers. The structure and computational complexity of
    the algorithm are unchanged. We emphasize that, since the weights
    \eqref{eq:SPOD-weights} are independent of $\eps$, so is the search
    criterion; hence a single set of generating matrices works for all
    perturbation values $\eps$.
\end{remark}

We now prove that the choice of Gevrey-SPOD weights \eqref{eq:SPOD-weights} yields a dimension-independent
QMC error bound. For a general Gevrey index $\mfs_{y} \ge 1$, this leads to an additional restriction on the admissible summability exponent $p$.

\begin{lemma}[Dimension-independent Gevrey-SPOD sum]
    \label{lem:gevrey-spod-summability}
    Let $\mfs_{y}\geq1$, let $0<p\leq1/\mfs_{y}$, and suppose that
    $\bm\beta\in\ell^p(\bN)$. Fix $C_2>0$ and set
    \begin{equation*}
        \alpha\coloneqq\floor{\frac1p}+1,
        \qquad \lambda\coloneqq p,
    \end{equation*}
    and let $\bgg$ be the Gevrey-SPOD weights
    \eqref{eq:SPOD-weights}. 
    For every prime $\mathfrak b$,
    \begin{equation}
        \label{eq:spod-summability-constant}
        S_{\mathrm{SPOD}}
        \coloneqq
        \sup_{s\in\bN}
        \sum_{\emptyset\neq\mfu\subseteq\{1:s\}}
        \gamma_{\mfu}^{p}
        [\rho_{\alpha,\mathfrak b}(p)]^{\abs{\mfu}}
        <\infty.
    \end{equation}
\end{lemma}

\begin{proof}
    Since $\alpha>1/p$, we have $1/\alpha<p\leq1$. More generally, fix
    $\lambda\in(1/\alpha,1]$ and define
    \begin{equation*}
        K_{\alpha,\mathfrak b}(\lambda)
        \coloneqq
        [\rho_{\alpha,\mathfrak b}(\lambda)]^{1/\lambda}.
    \end{equation*}
    Using $(\sum_i a_i)^\lambda\leq\sum_i a_i^\lambda$ for
    $a_i\geq0$ and $0<\lambda\leq1$, the SPOD weights satisfy
    \begin{align}
        &\sum_{\emptyset\neq\mfu\subseteq\{1:s\}}
        \gamma_{\mfu}^{\lambda}
        [\rho_{\alpha,\mathfrak b}(\lambda)]^{\abs{\mfu}}
        \nonumber\\
        &\quad\leq
        \sum_{\bm0\neq\bm\nu\in\{0:\alpha\}^s}
        (\abs{\bm\nu}!)^{\mfs_{y}\lambda}
        \prod_{\substack{1\leq j\leq s\\\nu_j>0}}
        \left(
        K_{\alpha,\mathfrak b}(\lambda)
        2^{\delta(\nu_j,\alpha)}
        C_2^{\nu_j}\beta_j^{\nu_j}
        \right)^{\lambda}.
        \label{eq:spod-reindexed-sum}
    \end{align}
    The zero components in $\{0:\alpha\}^s$ encode coordinates outside
    $\mfu$. Define
    \begin{equation}
        \label{eq:rescaled-qmc-weights}
        \widetilde\beta_j
        \coloneqq
        2C_2\max\{K_{\alpha,\mathfrak b}(\lambda),1\}\beta_j,
        \qquad j\geq1.
    \end{equation}
    For $1\leq r\leq\alpha$,
    \begin{equation*}
        K_{\alpha,\mathfrak b}(\lambda)
        2^{\delta(r,\alpha)}C_2^r\beta_j^r
        \leq\widetilde\beta_j^r.
    \end{equation*}
    Consequently, the right-hand side of
    \eqref{eq:spod-reindexed-sum} is bounded by
    \begin{equation}
        \label{eq:spod-multi-index-sum}
        T_s(\lambda)
        \coloneqq
        \sum_{\bm0\neq\bm\nu\in\{0:\alpha\}^s}
        \left(
        (\abs{\bm\nu}!)^{\mfs_{y}}
        \prod_{j=1}^s\widetilde\beta_j^{\nu_j}
        \right)^{\lambda}.
    \end{equation}
    
    Repeat each entry of $\widetilde{\bm\beta}$ exactly $\alpha$ times by setting
    \begin{equation}
        \label{eq:repeated-qmc-sequence}
        d_{\alpha(j-1)+r}\coloneqq\widetilde\beta_j,
        \qquad j\geq1,\quad r=1,\ldots,\alpha.
    \end{equation}
    Every term in \eqref{eq:spod-multi-index-sum} can then be represented by a
    finite subset $\mfv\subset\bN$. 
    Enlarging the sum to all finite subsets gives
    \begin{align}
        T_s(\lambda)
        &\leq
        \sum_{\substack{\emptyset\neq\mfv\subset\bN\\\abs{\mfv}<\infty}}
        \left(
        (\abs{\mfv}!)^{\mfs_{y}}
        \prod_{j\in\mfv}d_j
        \right)^{\lambda}
        =
        \sum_{\ell=1}^{\infty}
        (\ell!)^{\mfs_{y}\lambda}
        e_\ell(d_1^\lambda,d_2^\lambda,\ldots),
        \label{eq:spod-elementary-symmetric-sum}
    \end{align}
    where $e_\ell$ is the $\ell$th elementary symmetric sum. Since
    \begin{equation*}
        e_\ell(d_1^\lambda,d_2^\lambda,\ldots)
        \leq\frac1{\ell!}
        \left(\sum_{j=1}^{\infty}d_j^\lambda\right)^\ell,
    \end{equation*}
    we obtain
    \begin{equation}
        \label{eq:spod-factorial-series}
        T_s(\lambda)
        \leq
        \sum_{\ell=1}^{\infty}
        (\ell!)^{\mfs_{y}\lambda-1}D_\lambda^\ell,
        \qquad
        D_\lambda\coloneqq\sum_{j=1}^{\infty}d_j^\lambda.
    \end{equation}
    
    We now take $\lambda=p$. Since $\bm\beta\in\ell^p(\bN)$,
    \eqref{eq:rescaled-qmc-weights} and
    \eqref{eq:repeated-qmc-sequence} imply $D_p<\infty$. If
    $p<1/\mfs_{y}$, the ratio of consecutive terms on the right-hand side of
    \eqref{eq:spod-factorial-series} is
    \begin{equation*}
        D_p(\ell+1)^{\mfs_{y} p-1}\longrightarrow0,
    \end{equation*}
    so the series converges.
    
    It remains to consider the endpoint $p=1/\mfs_{y}$. Set $c_j\coloneqq d_j^p$.
    Then $\sum_{j\geq1}c_j<\infty$. Choose $J\in\bN$ such that
    \begin{equation*}
        T\coloneqq\sum_{j>J}c_j<1.
    \end{equation*}
    At the endpoint, $T_s(p)$ is bounded by
    $\sum_{\ell\geq1}\ell!e_\ell(c_1,c_2,\ldots)$. Splitting each elementary
    symmetric sum into its finite head and tail gives
    \begin{equation*}
        e_\ell(c_1,c_2,\ldots)
        =
        \sum_{r=0}^{\min\{J,\ell\}}
        e_r(c_1,\ldots,c_J)
        e_{\ell-r}(c_{J+1},c_{J+2},\ldots).
    \end{equation*}
    The tail satisfies
    \begin{equation*}
        e_k(c_{J+1},c_{J+2},\ldots)\leq\frac{T^k}{k!}.
    \end{equation*}
    Therefore
    \begin{align*}
        \sum_{\ell=0}^{\infty}\ell!e_\ell(c_1,c_2,\ldots)
        &\leq
        \sum_{r=0}^{J}e_r(c_1,\ldots,c_J)
        \sum_{k=0}^{\infty}\frac{(k+r)!}{k!}T^k
        <\infty,
    \end{align*}
    because $(k+r)!/k!$ grows polynomially in $k$ while $T^k$ decays
    geometrically. Thus the SPOD sum is bounded independently of $s$ also when
    $p=1/\mfs_{y}$, without a smallness condition on
    $\norm{\bm\beta}_{\ell^p(\bN)}$. This proves
    \eqref{eq:spod-summability-constant}.
\end{proof}

\begin{corollary}[Dimension-independent Gevrey QMC error]
    \label{cor:qmc-error-bound}
    Let $\mfs_{y}\geq1$, let $0<p<1$ satisfying $p \le 1/\mfs_{y}$, and suppose that
    $\bm\beta\in\ell^p(\bN)$. Fix $C_2>0$, set
    $\alpha=\floor{1/p}+1$ and let $\bgg$ be the Gevrey-SPOD weights
    \eqref{eq:SPOD-weights}. For every $s\in\bN$, every
    $1\leq q\leq\infty$, every prime $\mathfrak b$, and
    $N=\mathfrak b^m$, an interlaced polynomial lattice rule of order
    $\alpha$ can be constructed such that
    \begin{equation}
        \label{eq:cbc-bound-estimate-spod}
        \abs{Q_{N,s}(F)-I_s(F)}
        \leq
        C_{\mathrm{QMC}}
        \norm{F}_{s,\alpha,\bgg,q,\infty}N^{-1/p},
    \end{equation}
    where, with the constant $S_{\mathrm{SPOD}}$ 
    as in \eqref{eq:spod-summability-constant},
    \begin{equation}
        \label{eq:qmc-constant}
        C_{\mathrm{QMC}}
        \coloneqq(4S_{\mathrm{SPOD}})^{1/p}
    \end{equation}
    is independent of $s$ and $N$. In particular, if $F$ satisfies
    \eqref{eq:ass-QMC-integrand}, then
    \begin{equation}
        \label{eq:qmc-error-gevrey-integrand}
        \abs{Q_{N,s}(F)-I_s(F)}
        \leq C_{\mathrm{QMC}}C_1N^{-1/p}.
    \end{equation}
\end{corollary}

\begin{proof}
    Take $\lambda=p$ in \cref{thm:ho-qmc-error}. By
    \cref{lem:gevrey-spod-summability}, the sum in
    \eqref{eq:cbc-bound-estimate} is bounded by $S_{\mathrm{SPOD}}$. Since
    $N=\mathfrak b^m\geq2$, we have $N/(N-1)\leq2$, and hence
    \begin{equation*}
        \left(\frac{2S_{\mathrm{SPOD}}}{N-1}\right)^{1/p}
        \leq
        (4S_{\mathrm{SPOD}})^{1/p}N^{-1/p}.
    \end{equation*}
    This proves \eqref{eq:cbc-bound-estimate-spod}. The estimate
    \eqref{eq:qmc-error-gevrey-integrand} follows from
    \eqref{eq:gevrey-spod-norm-bound}.
\end{proof}

\begin{remark}
    The constant $C_{\mathrm{QMC}}$ depends on $p$, $\mfs_{y}$, the prime base
    $\mathfrak b$, $C_2$, and $\bm\beta$, but is independent of the dimension
    $s$ and the number of points $N$. The dependence on the particular
    integrand enters only through $C_1$ in
    \eqref{eq:ass-QMC-integrand}.
\end{remark}

\section{Combined error estimate}
\label{sec:combined-error}

We now specialize to $D=(-1,1)$ in order to use the robust spatial
approximation result from the appendix.  In addition to the assumptions above,
suppose that, for some spatial Gevrey index $\mfs_{x}\geq1$,
\begin{equation*}
    f\in\cG^{\mfs_x}(\overline D),
    \qquad
    b(\cdot;\bm y)\in\cG^{\mfs_x}(\overline D)
    \quad\forall\bm y\in U,
\end{equation*}
with Gevrey constants for $b$ that are uniform in $\bm y$.  More precisely,
suppose there are positive constants $C_{b,x}$, $\gamma_{b,x}$, $C_f$, and
$\gamma_f$ such that, for every $n\in\bN_0$,
\begin{equation}
    \label{eq:uniform-spatial-gevrey}
    \sup_{\bm y\in U}
    \norm{\partial_x^n b(\cdot;\bm y)}_{L^\infty(D)}
    \leq C_{b,x}\gamma_{b,x}^n(n!)^{\mfs_{x}},
    \qquad
    \norm{f^{(n)}}_{L^\infty(D)}
    \leq C_f\gamma_f^n(n!)^{\mfs_{x}}.
\end{equation}
The spatial index $\mfs_{x}$ need not coincide with the parametric index $\mfs_{y}$.
Thus every truncated coefficient
$b(\cdot;(\bm y_{\leq s},\bm0))$ satisfies the assumptions of
\cref{cor:gevrey-hp-galerkin}, with constants independent of $s$ and
$\bm y_{\leq s}$.  Fix an admissible mesh parameter $0<\kappa\leq\kappa_0$.
For a polynomial degree $p_{\mathrm{FE}}\in\bN$, let
$u_{\eps,s,p_{\mathrm{FE}}}$ denote the Galerkin approximation of
$u_{\eps,s}$ in
\begin{equation*}
    S_0(\kappa,p_{\mathrm{FE}})
    \coloneqq S(\kappa,p_{\mathrm{FE}})\cap H_0^1(D).
\end{equation*}

For $G\in L^2(D)$, identified with the functional
$v\mapsto(G,v)_{L^2(D)}$, set
\begin{equation*}
    \cI_\eps(G)
    \coloneqq\int_U G(u_\eps(\cdot;\bm y))\,\bm\pi(\dif\bm y).
\end{equation*}

\begin{theorem}[Combined QMC--$hp$-Galerkin error]
    \label{thm:combined-error}
    Suppose that
    \cref{ass:uniform-bounds,ass:gevrey-b,ass:dimension-truncation-weights}
    and \eqref{eq:uniform-spatial-gevrey} hold for some $0<p<1$ with
    $p\leq1/\mfs_{y}$.  Set $\alpha=\floor{1/p}+1$ and let $N=\mathfrak b^m$, where
    $\mathfrak b$ is prime.  Then the interlaced polynomial lattice rule from
    \cref{cor:qmc-error-bound}, translated to $[-1/2,1/2]^s$, satisfies, for every
    $\eps\in(0,1]$, $s,p_{\mathrm{FE}}\in\bN$, and $G\in L^2(D)$,
    \begin{equation}
        \label{eq:combined-error}
        \begin{aligned}
            \abs*{\cI_\eps(G)
                -Q_{N,s}\bigl(G(u_{\eps,s,p_{\mathrm{FE}}})\bigr)}
            &\leq C\Bigl(
            s^{-2/p+1}+N^{-1/p} \\
            &\qquad +\exp\bigl(-\sigma p_{\mathrm{FE}}^{1/\mfs_{x}}\bigr)
            \Bigr)\norm{G}_{L^2(D)}.
        \end{aligned}
    \end{equation}
    The constants $C,\sigma>0$ are independent of
    $\eps\in(0,1]$, $s$, $N$, and $p_{\mathrm{FE}}$.
\end{theorem}

\begin{proof}
    Insert the exact dimension-truncated integral and its QMC approximation.  The
    three resulting terms are, respectively, the dimension-truncation error, the
    QMC error, and the spatial Galerkin error.  The first is bounded by
    \cref{cor:dimension-truncation}.  By
    \cref{thm:gevrey-regularity}, the truncated integrand satisfies
    \eqref{eq:ass-QMC-integrand} uniformly in $\eps$ and $s$, so the second term is
    bounded by \cref{cor:qmc-error-bound}.  
    Finally, since the QMC rule has equal
    weights, \cref{cor:gevrey-hp-galerkin,rem:appendix-uniform-hp-constants} and
    $\norm{v}_{L^2(D)}\leq\norm{v}_{1,\eps,D}$ give
    \begin{equation*}
        \frac1N\sum_{n=0}^{N-1}
        \abs*{G\bigl(
            u_{\eps,s}(\cdot;\widehat{\bm y}_n)
            -u_{\eps,s,p_{\mathrm{FE}}}(\cdot;\widehat{\bm y}_n)
            \bigr)}
        \leq
        C\exp\bigl(-\sigma p_{\mathrm{FE}}^{1/\mfs_{x}}\bigr)
        \norm{G}_{L^2(D)}.
    \end{equation*}
    Here $\widehat{\bm y}_n\in[-1/2,1/2]^s$ denotes the translated QMC point.
    Combining these estimates proves \eqref{eq:combined-error}.
\end{proof}

\section{Numerical experiments}
\label{sec:numerical-experiments}
We perform numerical experiments to study the spatial discretization and
QMC errors in the QMC--$hp$-FE method.
Compared with \cite{RS26_1180}, the main new feature is a spatially
Gevrey-regular coefficient together with the approximation of the corresponding solutions by the boundary-layer adapted $hp$-FE spaces
analyzed in the appendix.  The parameter dependence remains affine and therefore corresponds to the analytic case 
$\mfs_{y}=1$ of the parametric regularity theory.  Recent work
\cite{chernov2026localgevrey,chernov2025analytic} on QMC for Gevrey-parametric diffusion problems
has observed essentially the same convergence rates as in the case of analytic parameter dependence. We therefore focus on the robust
spatial approximation and its interaction with higher-order QMC.
Throughout, we use $\eps\in\{10^{-3},10^{-5},10^{-7},10^{-9}\}$; the
corresponding convergence curves are nearly indistinguishable, confirming the $\eps$-robustness of the method.
\subsection{Model problem and discretization}
\label{subsec:num-model}
For $\bm y\in U=[-1/2,1/2]^{\bN}$, define the affine reaction coefficient
\begin{equation}
    \label{eq:num-reaction-coefficient}
    b(x;\bm y)
    =0.2+\sum_{j\geq1} y_j\psi_j(x),
    \quad
    \psi_j(x)=\beta_j\sin(j\pi x),
    \quad
    \beta_j=C_0\exp\bigl(-C_1j^{1/\mfs_{x}}\bigr),
\end{equation}
where $C_0=1$, $C_1=2$, and $\mfs_x=1.2$.  
This series is uniformly
$\mfs_x$-Gevrey in $x$: termwise differentiation and $\abs{y_j}\leq1/2$
give, for every $n\in\bN_0$,
\begin{align*}
    \sup_{\bm y\in U}
    \norm{\partial_x^n(b(\cdot;\bm y)-0.2)}_{L^\infty(D)}
    &\leq \frac{\pi^n}{2}\sum_{j\geq1}j^n e^{-2j^{1/\mfs_x}}\\
    &\leq \frac{\pi^n}{2}\int_0^\infty
    (x+1)^n e^{-2x^{1/\mfs_x}}\,\dif x \leq C\gamma^n(n!)^{\mfs_x}.
\end{align*}
Here the integral comparison is applied on each interval $[j-1,j]$; the final
bound follows from $(x+1)^n\leq2^n(1+x^n)$, the substitution
$t=2x^{1/\mfs_x}$, and Stirling's formula.  The constant term and all finite
truncations satisfy the same bound after adjusting $C$, uniformly in $\bm y$
and the truncation dimension.

For the computations, we fix $s=128$ and use the anchored truncation
\begin{equation*}
    b_s(x;\bm y_{\leq s})
    \coloneqq b(x;(\bm y_{\leq s},\bm0))
    =0.2+\sum_{j=1}^s y_j\psi_j(x),
    \quad \bm y_{\leq s}\in U_s=[-1/2,1/2]^s,
\end{equation*}
which satisfies
\begin{equation*}
    \sup_{\bm{y} \in U} \norm{b(\cdot;\bm{y})
        -b_s(\cdot;\bm y_{\leq s})}_{L^\infty(D)}
    \leq\frac12\sum_{j>s}\beta_j, 
    \quad \text{and} \quad
    b(x;\bm y)\geq 0.2 -\frac12\sum_{j\geq1}\beta_j>0.
\end{equation*}
Thus the truncation error is numerically negligible, and both the full and truncated
problems are uniformly coercive.  We suppress the fixed truncation index from
the solution and quantity-of-interest notation.  On $D=(-1,1)$ we solve
\begin{equation}
    \label{eq:num-model-problem}
    -\eps^2 u_\eps''(x;\bm y)+b_s(x;\bm y)u_\eps(x;\bm y)=1,
    \qquad
    u_\eps(-1;\bm y)=u_\eps(1;\bm y)=0.
\end{equation}

The quantity of interest is the spatial integral
\begin{equation}
    \label{eq:num-qoi}
    G(v)\coloneqq\int_{-1}^{1}v(x)\,\dif x,
    \quad
    F_\eps(\bm y)\coloneqq G(u_\eps(\cdot;\bm y)),
    \quad
    I_s(F_\eps)\coloneqq\int_{U_s}F_\eps(\bm y)\,\dif\bm{y}.
\end{equation}

For the spatial discretization we use an integrated-Legendre hierarchical
basis of uniform polynomial degree $p_{\mathrm{FE}}$.  With
$w=\kappa p_{\mathrm{FE}}\eps$ and $\kappa=1$, we use the three-element
boundary-layer mesh with nodes $\{-1,-1+w,1-w,1\}$ when $w<1/2$, and the single interval $[-1,1]$ otherwise.
Although an explicit admissible upper bound $\kappa_0$ could, in principle, be derived from the analysis in Appendix~\ref{sec:ExpCnv}, we do not attempt to determine such a bound here. Numerical experiments indicated that the convergence behavior is largely insensitive to the precise choice of $\kappa$. This is consistent with the observations reported in \cite{ScSu96,melenk1997boundarylayer}, where a similar problem was considered.

In all experiments, integrals are evaluated using composite Gauss--Legendre quadrature rules of sufficiently high order so that polynomial integrands are integrated exactly and all remaining quadrature errors are negligible; see also \cref{thm:appendix-quadrature}.

For the parametric discretization we use deterministic interlaced polynomial
lattice rules generated by the component-by-component construction in
the QMC4PDE package \cite{Nuyens2024QMC4PDE}. 
Separate point sets are constructed for the digit-interlacing factor $\alpha \in \{2,3,4\}$, with $N=2^m$
points, $m=2,\ldots,10$ and $m_{\mathrm{ref}}=12$, using the SPOD weights \eqref{eq:SPOD-weights} for the coefficient \eqref{eq:num-reaction-coefficient}. 
Following the numerical experiments in \cite[Section~4.1]{GantnerSchwabComputationalQMC}, 
we use the integrand-adapted Walsh-coefficient bound $C_{\mathrm{Walsh}}=1.0$. 
We emphasize that the construction of the lattice rules is independent of $\eps$, 
so each QMC-point set is reused for all values of the perturbation parameter $\eps$.

Since $(\beta_j)_{j\geq1}\in\ell^p(\bN)$ for every $p>0$, the QMC-theory permits
arbitrarily high-order rules. 
Due to practical limitations in the CBC construction of high-order lattice rules, 
we here focus on $\alpha=2,\ldots,4$ in the experiments.
\subsection{Robust \texorpdfstring{$hp$}{hp}-FE convergence}
\label{subsec:num-hp}
We first examine only the spatial discretization.  
We fix the parameter vector
\begin{equation*}
    \bm y_*=(1/2,\ldots,1/2)\in U_s
\end{equation*}
and compute the $hp$-FE solution for
\begin{equation*}
    p_{\mathrm{FE}}\in\{2,4,8,16,24,32,40,48,56,64,72,80\}.
\end{equation*}
The solution with $p_{\mathrm{FE,ref}}=96$ serves as reference. 
For
$e_{p_{\mathrm{FE}}}=u_{\eps,p_{\mathrm{FE}}}(\cdot;\bm y_*)-
u_{\eps,p_{\mathrm{FE,ref}}}(\cdot;\bm y_*)$, 
we measure the relative error
\begin{equation*}
    \frac{\norm{e_{p_{\mathrm{FE}}}}_{1,\eps,D}}
    {\norm{u_{\eps,p_{\mathrm{FE,ref}}}(\cdot;\bm y_*)}_{1,\eps,D}},
    \quad\mbox{where}\;\;
    \norm{v}_{1,\eps,D}^2
    =\eps^2\norm{v'}_{L^2(D)}^2+\norm{v}_{L^2(D)}^2.
\end{equation*}
The predicted spatial rate has the form
\begin{equation}
    \label{eq:num-hp-rate}
    \norm{e_{p_{\mathrm{FE}}}}_{1,\eps,D}
    \simeq C\exp\bigl(-\sigma p_{\mathrm{FE}}^{1/\mfs_{x}}\bigr).
\end{equation}
An empirical fit gives a numerical value of $\sigma$ in this error bound of
\begin{equation}
    \label{eq:num-sigma-est}
    \sigma_{\mathrm{est}}=0.827.
\end{equation}
This value is used only to balance the two error contributions in
\cref{subsec:num-combined}, and numerical evidence suggests that it is independent of $\bm{y}_{*}$.

Panel (a) of \cref{fig:convergence-plot-1} displays the relative error versus the number of spatial degrees of freedom.  
Its rapid decay is consistent with
\eqref{eq:num-hp-rate}, while the four nearly identical curves demonstrate
robustness with respect to $\eps$.

\subsection{Higher-order QMC convergence}
\label{subsec:num-qmc}

We next isolate the QMC error by fixing the spatial degree at $p_{\mathrm{FE}}=96$, for
which the spatial discretization error is negligible on the scale of this
experiment.  For each interlacing factor $\alpha$ and each $N=2^m$, for $m=2,\dots,10$, we compute
\begin{equation}
    \label{eq:num-qmc-estimator}
    F_{\eps,p_{\mathrm{FE}}}(\bm y)
    \coloneqq G(u_{\eps,p_{\mathrm{FE}}}(\cdot;\bm y)),
    \qquad
    Q_{N,s}(F_{\eps,p_{\mathrm{FE}}})
    =\frac1N\sum_{n=0}^{N-1}
    F_{\eps,p_{\mathrm{FE}}}(\widehat{\bm y}_n).
\end{equation}
For each interlacing factor $\alpha$, a reference value is computed using $N_{\mathrm{ref}}=2^{12}=4096$ QMC points. We compute the absolute
errors with respect to the reference values,
\begin{equation*}
    \abs{Q_{N,s}(F_{\eps,p_{\mathrm{FE}}})
        -Q_{N_{\mathrm{ref}},s}(F_{\eps,p_{\mathrm{FE}}})}.
\end{equation*}
For the affine parameter dependence considered here, the theory yields, for an interlacing factor $\alpha$, the asymptotic convergence rate
\begin{equation}
    \label{eq:num-qmc-rate}
    \abs{Q_{N,s}(F_{\eps,p_{\mathrm{FE}}})
        -I_s(F_{\eps,p_{\mathrm{FE}}})}
    =\cO(N^{-(\alpha-\delta)}),
\end{equation}
for arbitrarily small $\delta>0$.
The results for $\alpha=2,3,4$ are shown in
\cref{fig:convergence-plot-1}, panels (b)--(d).
Over the range considered, the numerical results are consistent with the
asymptotic convergence rates in \eqref{eq:num-qmc-rate}. We also tested
higher-order interlaced rules with $\alpha\ge 5$, for which the numerical
behavior no longer clearly followed the predicted rates. This may be
attributable to an extended, pre-asymptotic regime 
so that the computational range
accessible here is insufficient to reveal the asymptotic rate, or to numerical
difficulties in the CBC construction at higher interlacing factors. Such
difficulties in practical CBC constructions have also been reported in the
literature; see, e.g.,
\cite{GantnerSchwabComputationalQMC,kritzer2016projection}.
For every $\alpha$, the four values of $\eps$ produce nearly identical curves,
again confirming numerically the robustness of the error bounds
with respect to $\eps$.
\begin{figure}
    \centering
    \includegraphics[width=\textwidth]{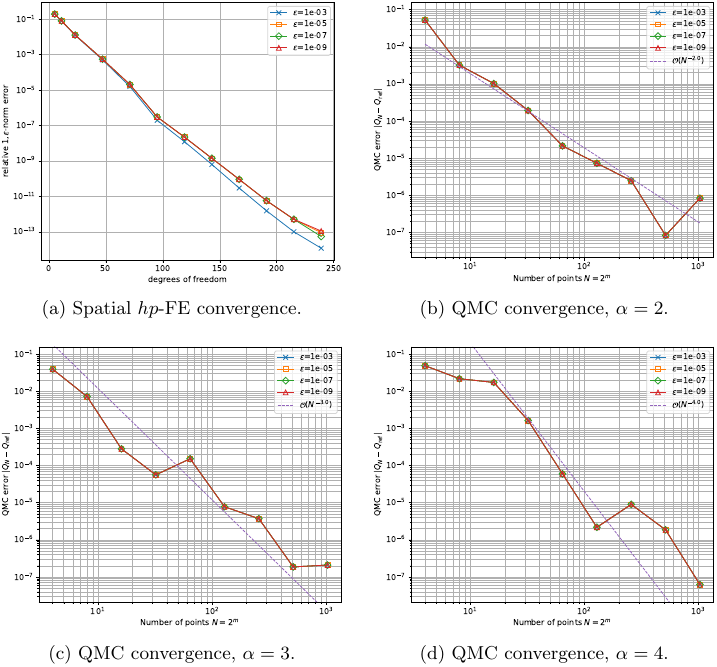}
    \caption{Convergence of the spatial and parametric discretizations
        for the four values of $\eps$.
        Panel (a) shows the relative
        $\norm{\cdot}_{1,\eps,D}$-error versus the number of spatial degrees
        of freedom.
        Panels (b)--(d) show the absolute QMC errors for digit-interlacing
        factors $\alpha=2,3,4$, measured with respect to the corresponding
        reference values computed with $p_{\mathrm{FE,ref}}=96$ and
        $N_{\mathrm{ref}}=4096$.}
    \label{fig:convergence-plot-1}
\end{figure}

\clearpage

\subsection{QMC--\texorpdfstring{$hp$}{hp}-FE convergence}
\label{subsec:num-combined}
Finally, we refine the QMC rule and spatial discretization simultaneously.
For this experiment we use interlacing factors $\alpha=2,3,4$. Balancing the (asymptotic) error contributions
$N^{-\alpha}$ and
$\exp(-\sigma_{\mathrm{est}}p_{\mathrm{FE}}^{1/\mfs_{x}})$ gives the relation
\begin{equation}
    p_{\mathrm{FE}}(N)
    =\left\lceil
    \left(\frac{\alpha \log N}{\sigma_{\mathrm{est}}}\right)^{\mfs_{x}}
    \right\rceil,
    \qquad \sigma_{\mathrm{est}}=0.827.
    \label{eq:num-balanced-schedule}
\end{equation}
For $N=2^m$, $m=2,\ldots,9$, we compute
\begin{equation*}
    Q_{N,s}(F_{\eps,p_{\mathrm{FE}}(N)})
    =\frac1N\sum_{n=0}^{N-1}
    F_{\eps,p_{\mathrm{FE}}(N)}(\widehat{\bm y}_n).
\end{equation*}
A reference $Q_{N_{\mathrm{ref}},s}(F_{\eps,p_{\mathrm{FE,ref}}})$ is evaluated with $N_{\mathrm{ref}}=2^{10}=1024$ and $p_{\mathrm{FE,ref}}=96$.

We use the work model
\begin{equation}
    \operatorname{work}(N,p_{\mathrm{FE}})=N \cdot p_{\mathrm{FE}}^3.
    \label{eq:num-work-model}
\end{equation}
The total errors
\begin{equation*}
    \abs{Q_{N,s}(F_{\eps,p_{\mathrm{FE}}(N)})
        -Q_{N_{\mathrm{ref}},s}(F_{\eps,p_{\mathrm{FE,ref}}})}
\end{equation*}
are plotted versus the work model \eqref{eq:num-work-model} in \cref{fig:convergence-plot-2}.

For comparison, we also include error-versus-work results for the standard Monte Carlo method combined with $hp$-FE discretization. 
We use the same reference values as in the QMC experiments and average the errors over ten independent realizations to reduce the sample-variance of the Monte Carlo estimates. 
The higher-order QMC method is deterministic and therefore requires no such averaging. 
The sample numbers are similarly chosen so as to balance 
the asymptotic error contributions subject to the corresponding work bounds. Due to the robust root-exponential convergence rate of $hp$-FEM, the error of the MC--$hp$-FE approximation is dominated by the rate $1/2$ of MC sampling. 
For each method, we also plot reference curves evaluated at the same work values as the corresponding error curves, with coordinates $(N p_{\mathrm{FE}}^3,N^{-\alpha})$ for QMC and $(N p_{\mathrm{FE}}^3,N^{-1/2})$ for Monte Carlo.

The QMC--$hp$-FE approximations show substantially faster error decay 
with respect to work than the standard MC--$hp$-FE method, 
consistent with the theoretical convergence rates. 
For each of $\alpha=2,3,4$, 
the four values of $\eps$ produce nearly identical curves, 
again indicating robustness with respect to $\eps$.

\begin{figure}
    \centering
    \includegraphics[width=\textwidth]{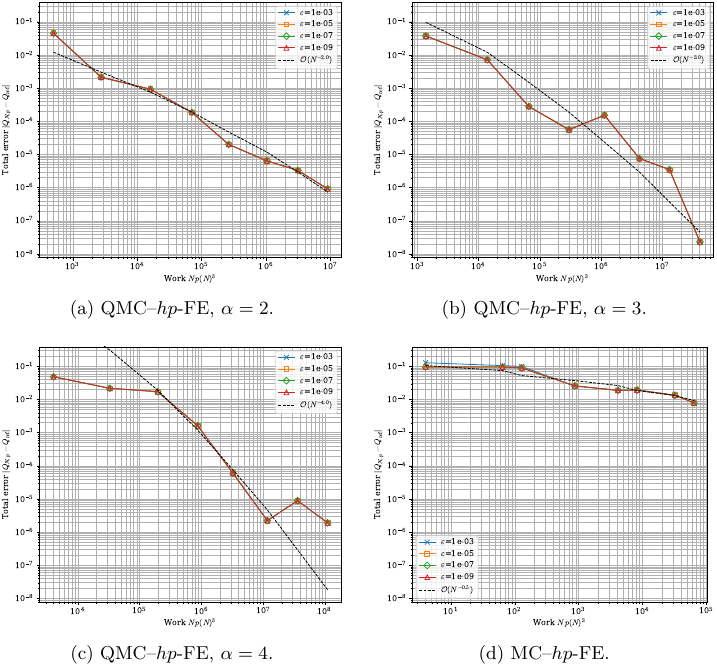}
    \caption{Combined QMC--$hp$-FE error versus the work model
        $Np_{\mathrm{FE}}^3$, with $N$ and $p_{\mathrm{FE}}$ chosen according to
        \eqref{eq:num-balanced-schedule} to balance the QMC and $hp$-FE error
        contributions, for digit-interlacing factors
        $\alpha\in\{2,3,4\}$, together with an MC--$hp$-FE baseline.
        The corresponding reference values are computed with
        $N_{\mathrm{ref}}=1024$ QMC points and
        $p_{\mathrm{FE,ref}}=96$.}
    \label{fig:convergence-plot-2}
\end{figure}
\clearpage
\section{Conclusion}
\label{sec:conclusion}

We analyzed the combined $hp$-FE and Higher Order QMC 
discretization of singularly perturbed
reaction--diffusion problems with uncertain, parametric
reaction coefficient $b$ and fixed source term $f$. 
Under weighted mixed Gevrey-$\mfs_{y}$ regularity assumptions on $b$, with 
anisotropy weights $\bm{\beta} \in \ell^p(\bN)$ for some $0<p<1$ satisfying
$p\leq1/\mathfrak{s}_y$, and under spatial Gevrey-$\mathfrak{s}_x$ regularity
assumptions on $b$ and $f$, 
we proved a robust, i.e., independent of the singular perturbation
parameter $0<\eps\leq 1$,
root-exponential convergence rate of the $hp$-FEM.
We also established an $\eps$-uniform dimension truncation error bound
for the parametric solution. 
The algebraic convergence rate $1/p$ of HoQMC is not subject to the CoD,
and is likewise uniform with respect to $\eps$.

\printbibliography

@article{Bakhvalov1969,
    title = {The optimization of methods of solving boundary value problems with a boundary layer},
    journal = {USSR Computational Mathematics and Mathematical Physics},
    volume = {9},
    number = {4},
    pages = {139-166},
    year = {1969},
    issn = {0041-5553},
    doi = {https://doi.org/10.1016/0041-5553(69)90038-X},
    url = {https://www.sciencedirect.com/science/article/pii/004155536990038X},
    author = {N.S. Bakhvalov},
}

@article {Dick2007Explicit,
	AUTHOR = {Dick, Josef},
	TITLE = {Explicit constructions of quasi-{M}onte {C}arlo rules for the
	         numerical integration of high-dimensional periodic functions},
	JOURNAL = {SIAM J. Numer. Anal.},
	FJOURNAL = {SIAM Journal on Numerical Analysis},
	VOLUME = {45},
	YEAR = {2007},
	NUMBER = {5},
	PAGES = {2141--2176},
	ISSN = {0036-1429,1095-7170},
	MRCLASS = {11K38 (11K45 65D30)},
	MRNUMBER = {2346374},
	MRREVIEWER = {Erich\ Novak},
	DOI = {10.1137/060658916},
	URL = {https://doi.org/10.1137/060658916},
}

@article {Dick2008Walsh,
	AUTHOR = {Dick, Josef},
	TITLE = {Walsh spaces containing smooth functions and quasi-{M}onte
	         {C}arlo rules of arbitrary high order},
	JOURNAL = {SIAM J. Numer. Anal.},
	FJOURNAL = {SIAM Journal on Numerical Analysis},
	VOLUME = {46},
	YEAR = {2008},
	NUMBER = {3},
	PAGES = {1519--1553},
	ISSN = {0036-1429,1095-7170},
	MRCLASS = {42C10 (11K38 11K45 65C05)},
	MRNUMBER = {2391005},
	MRREVIEWER = {Friedrich\ Pillichshammer},
	DOI = {10.1137/060666639},
	URL = {https://doi.org/10.1137/060666639},
}

@article {dick2014higher,
	AUTHOR = {Dick, Josef and Kuo, Frances Y. and Le Gia, Quoc T. and Nuyens, Dirk and Schwab, Christoph},
	TITLE = {Higher order {QMC} {P}etrov-{G}alerkin discretization for
	affine parametric operator equations with random field inputs},
	JOURNAL = {SIAM J. Numer. Anal.},
	FJOURNAL = {SIAM Journal on Numerical Analysis},
	VOLUME = {52},
	YEAR = {2014},
	NUMBER = {6},
	PAGES = {2676--2702},
	ISSN = {0036-1429,1095-7170},
	MRCLASS = {65D30 (65C05 65N30)},
	MRNUMBER = {3276428},
	MRREVIEWER = {Petru\ P.\ Blaga},
	DOI = {10.1137/130943984},
	URL = {https://doi.org/10.1137/130943984},
}

@article {dick2016multilevel,
	AUTHOR = {Dick, Josef and Kuo, Frances Y. and Le Gia, Quoc T. and Schwab, Christoph},
	TITLE = {Multilevel higher order {QMC} {P}etrov-{G}alerkin
	         discretization for affine parametric operator equations},
	JOURNAL = {SIAM J. Numer. Anal.},
	FJOURNAL = {SIAM Journal on Numerical Analysis},
	VOLUME = {54},
	YEAR = {2016},
	NUMBER = {4},
	PAGES = {2541--2568},
	ISSN = {0036-1429,1095-7170},
	MRCLASS = {65C05 (65D30 65N30 65N75)},
	MRNUMBER = {3537887},
	MRREVIEWER = {Bamdad\ Hosseini},
	DOI = {10.1137/16M1078690},
	URL = {https://doi.org/10.1137/16M1078690},
}

@article {dick2016holomorphic,
	AUTHOR = {Dick, Josef and Le Gia, Quoc T. and Schwab, Christoph},
	TITLE = {Higher order quasi-{M}onte {C}arlo integration for
	         holomorphic, parametric operator equations},
	JOURNAL = {SIAM/ASA J. Uncertain. Quantif.},
	FJOURNAL = {SIAM/ASA Journal on Uncertainty Quantification},
	VOLUME = {4},
	YEAR = {2016},
	NUMBER = {1},
	PAGES = {48--79},
	ISSN = {2166-2525},
	MRCLASS = {65D30 (65C05)},
	MRNUMBER = {3659685},
	MRREVIEWER = {Jiu\ Ding},
	DOI = {10.1137/140985913},
	URL = {https://doi.org/10.1137/140985913},
}

@book {dick2010digital,
	AUTHOR = {Dick, Josef and Pillichshammer, Friedrich},
	TITLE = {Digital nets and sequences},
	NOTE = {Discrepancy theory and quasi-Monte Carlo integration},
	PUBLISHER = {Cambridge University Press, Cambridge},
	YEAR = {2010},
	PAGES = {xviii+600},
	ISBN = {978-0-521-19159-3},
	MRCLASS = {65C05 (11K06 11K38 11K45 46E35 65D30)},
	MRNUMBER = {2683394},
	MRREVIEWER = {Petru\ P.\ Blaga},
	DOI = {10.1017/CBO9780511761188},
	URL = {https://doi.org/10.1017/CBO9780511761188},
}

@incollection {gantner2018dimension,
	AUTHOR = {Gantner, Robert N.},
	TITLE = {Dimension truncation in {QMC} for affine-parametric operator equations},
	BOOKTITLE = {Monte {C}arlo and quasi--{M}onte {C}arlo methods},
	SERIES = {Springer Proc. Math. Stat.},
	VOLUME = {241},
	PAGES = {249--264},
	PUBLISHER = {Springer, Cham},
	YEAR = {2018},
	ISBN = {978-3-319-91436-7; 978-3-319-91435-0},
	MRCLASS = {65C05},
	MRNUMBER = {3828144},
	DOI = {10.1007/978-3-319-91436-7\_13},
	URL = {https://doi.org/10.1007/978-3-319-91436-7_13},
}

@incollection {GantnerSchwabComputationalQMC,
	AUTHOR = {Gantner, Robert N. and Schwab, Christoph},
	TITLE = {Computational higher order quasi-{M}onte {C}arlo integration},
	BOOKTITLE = {Monte {C}arlo and quasi-{M}onte {C}arlo methods},
	SERIES = {Springer Proc. Math. Stat.},
	VOLUME = {163},
	PAGES = {271--288},
	PUBLISHER = {Springer, [Cham]},
	YEAR = {2016},
	ISBN = {978-3-319-33507-0; 978-3-319-33505-6},
	MRCLASS = {65D30 (65C05)},
	MRNUMBER = {3536664},
	DOI = {10.1007/978-3-319-33507-0\_12},
	URL = {https://doi.org/10.1007/978-3-319-33507-0_12},
}

@incollection {guth2022application,
	AUTHOR = {Guth, Philipp A. and Kaarnioja, Vesa},
	TITLE = {Application of dimension truncation error analysis to
	high-dimensional function approximation in uncertainty quantification},
	BOOKTITLE = {Monte {C}arlo and quasi-{M}onte {C}arlo methods},
	SERIES = {Springer Proc. Math. Stat.},
	VOLUME = {460},
	PAGES = {297--312},
	PUBLISHER = {Springer, Cham},
	YEAR = {2024},
	ISBN = {978-3-031-59761-9; 978-3-031-59762-6},
	MRCLASS = {65D40 (65C20)},
	MRNUMBER = {4774926},
	DOI = {10.1007/978-3-031-59762-6\_14},
	URL = {https://doi.org/10.1007/978-3-031-59762-6_14},
}

@article {guth2024lognormal,
	AUTHOR = {Guth, Philipp A. and Kaarnioja, Vesa},
	TITLE = {Generalized dimension truncation error analysis for
	high-dimensional numerical integration: lognormal setting and beyond},
	JOURNAL = {SIAM J. Numer. Anal.},
	FJOURNAL = {SIAM Journal on Numerical Analysis},
	VOLUME = {62},
	YEAR = {2024},
	NUMBER = {2},
	PAGES = {872--892},
	ISSN = {0036-1429,1095-7170},
	MRCLASS = {65D32 (35J15 35R60 65D30)},
	MRNUMBER = {4724084},
	DOI = {10.1137/23M1593188},
	URL = {https://doi.org/10.1137/23M1593188},
}

@article {kuo2016survey,
	AUTHOR = {Kuo, Frances Y. and Nuyens, Dirk},
	TITLE = {Application of quasi-{M}onte {C}arlo methods to elliptic
	{PDE}s with random diffusion coefficients: a survey of analysis and implementation},
	JOURNAL = {Found. Comput. Math.},
	FJOURNAL = {Foundations of Computational Mathematics. The Journal of the
	Society for the Foundations of Computational Mathematics},
	VOLUME = {16},
	YEAR = {2016},
	NUMBER = {6},
	PAGES = {1631--1696},
	ISSN = {1615-3375,1615-3383},
	MRCLASS = {65C05 (60H35 65N30 65N75)},
	MRNUMBER = {3579719},
	MRREVIEWER = {Romeo\ I.\ Negrea},
	DOI = {10.1007/s10208-016-9329-5},
	URL = {https://doi.org/10.1007/s10208-016-9329-5},
}

@article {kuo2012QMC,
	AUTHOR = {Kuo, Frances Y. and Schwab, Christoph and Sloan, Ian H.},
	TITLE = {Quasi-{M}onte {C}arlo finite element methods for a class of
	elliptic partial differential equations with random coefficients},
	JOURNAL = {SIAM J. Numer. Anal.},
	FJOURNAL = {SIAM Journal on Numerical Analysis},
	VOLUME = {50},
	YEAR = {2012},
	NUMBER = {6},
	PAGES = {3351--3374},
	ISSN = {0036-1429,1095-7170},
	MRCLASS = {65N30 (60H25 65C05 65N75)},
	MRNUMBER = {3024159},
	MRREVIEWER = {R\'emi\ Vaillancourt},
	DOI = {10.1137/110845537},
	URL = {https://doi.org/10.1137/110845537},
}

@article {melenk1997boundarylayer,
	AUTHOR = {Melenk, Jens Markus},
	TITLE = {On the robust exponential convergence of {$hp$} finite element
	method for problems with boundary layers},
	JOURNAL = {IMA J. Numer. Anal.},
	FJOURNAL = {IMA Journal of Numerical Analysis},
	VOLUME = {17},
	YEAR = {1997},
	NUMBER = {4},
	PAGES = {577--601},
	ISSN = {0272-4979,1464-3642},
	MRCLASS = {65N30 (65N12)},
	MRNUMBER = {1476340},
	MRREVIEWER = {Bernard\ Brighi},
	DOI = {10.1093/imanum/17.4.577},
	URL = {https://doi.org/10.1093/imanum/17.4.577},
}

@misc{Nuyens2024QMC4PDE,
	author       = {Dirk Nuyens},
	title        = {QMC4PDE: Quasi-Monte Carlo for PDEs},
	year         = {2024},
	howpublished = {\url{https://people.cs.kuleuven.be/~dirk.nuyens/qmc4pde/}},
	note         = {Version dated 10 December 2024}
}

@report{RS26_1180,
  author = {F. Rohner and C. Schwab},
  title = {Higher-Order QMC Galerkin FEM for Singular Perturbation:
Reaction--Diffusion with Affine-Parametric Reaction},
  institution = {Seminar for Applied Mathematics, ETH Z{\"u}rich},
  number = {2026-25},
  address = {Switzerland},
  url = {https://www.sam.math.ethz.ch/sam_reports/reports_final/reports2026/2026-25.pdf },
  year = {2026}
}

@report{ManuskriptDLBL,
	author = {F. Rohner and Ch. Schwab and C. Xenophontos},
	title = {Neural Networks for Singular Perturbations - Finite Regularity},
	institution = {Seminar for Applied Mathematics, ETH Z{\"u}rich},
	number = {2026-17},
	address = {Switzerland},
	url = {https://www.sam.math.ethz.ch/sam_reports/reports_final/reports2026/2026-17.pdf },
	year = {2026}
}

@book{RST2nd,
	AUTHOR = {Roos, Hans-G\"{o}rg and Stynes, Martin and Tobiska, Lutz},
	TITLE = {Robust numerical methods for singularly perturbed differential equations},
	SERIES = {Springer Series in Computational Mathematics},
	VOLUME = {24},
	EDITION = {Second Ed.},
	NOTE = {Convection-diffusion-reaction and flow problems},
	PUBLISHER = {Springer-Verlag, Berlin},
	YEAR = {2008},
	PAGES = {xiv+604},
	ISBN = {978-3-540-34466-7},
	MRCLASS = {65-02 (35B25 65Lxx 65Mxx 65Nxx 76M10 76V05)},
	MRNUMBER = {2454024},
	MRREVIEWER = {Eugene O'Riordan},
}

@article {Shishkin1989,
    AUTHOR = {Shishkin, G. I.},
    TITLE = {Grid approximation of singularly perturbed boundary value
    problems with a regular boundary layer},
    JOURNAL = {Soviet J. Numer. Anal. Math. Modelling},
    FJOURNAL = {Soviet Journal of Numerical Analysis and Mathematical
    Modelling},
    VOLUME = {4},
    YEAR = {1989},
    NUMBER = {5},
    PAGES = {397--417},
    ISSN = {0169-2895},
    MRCLASS = {65N99},
    MRNUMBER = {1026911},
    MRREVIEWER = {V.\ F.\ Butuzov},
}

@article {chernov2025analytic,
	AUTHOR = {Chernov, Alexey and L\^e, T\`ung},
	TITLE = {Analytic and {G}evrey class regularity for parametric
	semilinear reaction-diffusion problems and applications in
	uncertainty quantification},
	JOURNAL = {Comput. Math. Appl.},
	FJOURNAL = {Computers \& Mathematics with Applications. An International
	Journal},
	VOLUME = {164},
	YEAR = {2024},
	PAGES = {116--130},
	ISSN = {0898-1221,1873-7668},
	MRCLASS = {65N30 (35B65 35K57 65C05)},
	MRNUMBER = {4736610},
	DOI = {10.1016/j.camwa.2024.04.007},
	URL = {https://doi.org/10.1016/j.camwa.2024.04.007},
}

@article{chernov2026localgevrey,
	title={Local Gevrey regularity and Quasi-Monte Carlo quadrature for PDEs parameterized on non-compact domains}, 
	author={Dinh Dũng and Tùng Lê and Alexey Chernov},
	year={2026},
	eprint={2606.30972},
	archivePrefix={arXiv},
	primaryClass={math.NA},
	url={https://arxiv.org/abs/2606.30972}, 
}

@article {feischl2020exponential,
	AUTHOR = {Feischl, M. and Schwab, Ch.},
	TITLE = {Exponential convergence in {$H^1$} of {$hp$}-{FEM} for
	{G}evrey regularity with isotropic singularities},
	JOURNAL = {Numer. Math.},
	FJOURNAL = {Numerische Mathematik},
	VOLUME = {144},
	YEAR = {2020},
	NUMBER = {2},
	PAGES = {323--346},
	ISSN = {0029-599X,0945-3245},
	MRCLASS = {65N30},
	MRNUMBER = {4057425},
	MRREVIEWER = {Ariel\ L.\ Lombardi},
	DOI = {10.1007/s00211-019-01085-z},
	URL = {https://doi.org/10.1007/s00211-019-01085-z},
}

@article {costabel2005maxwell,
	AUTHOR = {Costabel, Martin and Dauge, Monique and Schwab, Christoph},
	TITLE = {Exponential convergence of {$hp$}-{FEM} for {M}axwell
	equations with weighted regularization in polygonal domains},
	JOURNAL = {Math. Models Methods Appl. Sci.},
	FJOURNAL = {Mathematical Models and Methods in Applied Sciences},
	VOLUME = {15},
	YEAR = {2005},
	NUMBER = {4},
	PAGES = {575--622},
	ISSN = {0218-2025,1793-6314},
	MRCLASS = {78M10 (65N30)},
	MRNUMBER = {2137526},
	MRREVIEWER = {Snorre\ H.\ Christiansen},
	DOI = {10.1142/S0218202505000480},
	URL = {https://doi.org/10.1142/S0218202505000480},
}

@article {sundermann1983lebesgue,
	AUTHOR = {S\"undermann, Burkhard},
	TITLE = {Lebesgue constants in {L}agrangian interpolation at the {F}ekete points},
	JOURNAL = {Mitt. Math. Ges. Hamburg},
	FJOURNAL = {Mitteilungen der Mathematischen Gesellschaft in Hamburg},
	VOLUME = {11},
	YEAR = {1983},
	NUMBER = {2},
	PAGES = {204--211},
	ISSN = {0340-4358},
	MRCLASS = {41A05},
	MRNUMBER = {744527},
	MRREVIEWER = {Ambikeswar\ Sharma},
}

@article{ScSu96,
    AUTHOR = {Schwab, Christoph and Suri, Manil},
     TITLE = {The {$p$} and {$hp$} versions of the finite element method for problems with boundary layers},
   JOURNAL = {Math. Comp.},
  FJOURNAL = {Mathematics of Computation},
    VOLUME = {65},
      YEAR = {1996},
    NUMBER = {216},
     PAGES = {1403--1429},
      ISSN = {0025-5718,1088-6842},
   MRCLASS = {65L60 (65N30)},
  MRNUMBER = {1370857},
MRREVIEWER = {Josef\ Dal\'{\i}k},
       DOI = {10.1090/S0025-5718-96-00781-8},
       URL = {https://doi.org/10.1090/S0025-5718-96-00781-8},
}

@article {kritzer2016projection,
	AUTHOR = {Dick, Josef and Kritzer, Peter},
	TITLE = {On a projection-corrected component-by-component construction},
	JOURNAL = {J. Complexity},
	FJOURNAL = {Journal of Complexity},
	VOLUME = {32},
	YEAR = {2016},
	NUMBER = {1},
	PAGES = {74--80},
	ISSN = {0885-064X,1090-2708},
	MRCLASS = {65D30 (65C05)},
	MRNUMBER = {3424638},
	MRREVIEWER = {S.\ Hitotumatu},
	DOI = {10.1016/j.jco.2015.08.001},
	URL = {https://doi.org/10.1016/j.jco.2015.08.001},
}

\appendix
\section{Robust (sub)-exponential convergence of the \texorpdfstring{$hp$}{hp}-FEM
    for singularly perturbed, elliptic BVPs under Gevrey regularity}
\label{sec:ExpCnv}
This appendix extends the robust exponential convergence analysis of 
\cite{melenk1997boundarylayer} from analytic to $\mfs_{x}$-Gevrey input data, 
i.e. to $\mfs_{x} > 1$, where we show $\eps$-robust, root-exponential rates.
We follow \cite{melenk1997boundarylayer} closely, 
both in the organization of the material and in the notation, 
so that the results in \cite{melenk1997boundarylayer} 
are recovered for $\mfs_{x}=1$.

In order to keep the exposition concise, 
\emph{we only detail those arguments
    which differ mathematically from}
the corresponding proof in \cite{melenk1997boundarylayer}. 
All remaining statements are proved exactly as
their analytic counterparts in \cite{melenk1997boundarylayer}, 
upon the two routine substitutions
\begin{itemize}
    \item replacing the analytic factorial growth $n!$ by its $\mfs_{x}$-Gevrey
    analogue $(n!)^{\mfs_{x}}$ (equivalently, a bound $q^{p}$ on a rate is replaced
    by $q^{p^{1/\mfs_{x}}}$), and
    \item replacing every appeal to an analytic auxiliary result of
    \cite{melenk1997boundarylayer} by its Gevrey analogue stated below.
\end{itemize}
For such statements we only indicate the corresponding result in
\cite{melenk1997boundarylayer}, and we do not reproduce the proof.
\subsection{Model Problem}
\label{sec:ModPrb}
This appendix concerns the $p$- and $hp$-version finite element approximation of
the singularly perturbed model boundary value problem
\begin{equation}\label{eq:appendix-problem}
    \begin{aligned}
        L_{\eps} u_{\eps} \coloneqq -\eps^2 u_\eps'' + b(x) u_\eps
        &=f, \quad \text{on } \Omega \coloneqq (-1,1),\\
        u_\eps(\pm 1)
        &=\alpha^{\pm} \in \mathbb{R}.
    \end{aligned}
\end{equation}
Here $\mfs_x\geq1$ and
$f,b\in\cG^{\mfs_x}(\overline\Omega)$ in the sense of
\eqref{eq:gevrey-class-finite-dimensional}.  The reaction coefficient is
assumed to satisfy $b(x)\geq\underline b^2>0$ on $\overline\Omega$, and
$\eps\in(0,1]$ is the singular perturbation parameter, which may be
arbitrarily close to zero.  We fix constants
$C_b,\gamma_b,C_f,\gamma_f>0$ quantifying the Gevrey regularity such that
\begin{align}
    \norm{b^{(n)}}_{L^{\infty}(\Omega)} &\le C_{b} \gamma_{b}^n (n!)^{\mfs_{x}} \qquad \forall n \in \mathbb{N}_0,\\
    \norm{f^{(n)}}_{L^{\infty}(\Omega)} &\le C_{f} \gamma_f^n (n!)^{\mfs_{x}} \qquad \forall n \in \mathbb{N}_0.
\end{align}
It is this regularity that will allow us to control the asymptotic expansion of 
$u_{\eps}$ explicitly in $\eps$ and in the expansion truncation order.

The variational formulation of \eqref{eq:appendix-problem} is
\begin{equation}
    \label{eq:appendix-weak-form}
    \text{find } u_{\eps} \in H^{1}_{D}(\Omega) \text{ such that} \qquad
    B_{\eps}(u_{\eps},v) = F(v) \qquad \forall v \in H^{1}_{0}(\Omega),
\end{equation}
where we set
\begin{equation*}
    B_{\eps}(u,v) = \int_{\Omega} \bigl(\eps^2 u' v' + b(x) u v\bigr) \dif x,
    \qquad
    F(v) = \int_{\Omega} f v \dif x,
\end{equation*}
and where
\begin{equation*}
    H^{1}_{D}(\Omega) \coloneqq \{u \in H^{1}(\Omega) : u(-1) = \alpha^{-},\, u(1) = \alpha^{+}\},
    \quad 
    H^{1}_{0}(\Omega) \coloneqq \{u \in H^{1}(\Omega) : u(\pm 1) = 0\}.
\end{equation*}
As usual, $H^{1}(\Omega)$ is the Sobolev space of $L^2(\Omega)$-functions 
whose distributional derivative is again square integrable.
The bilinear form $B_{\eps}$ induces on $H^{1}(\Omega)$ the energy norm
\begin{equation}
    \label{eq:appendix-energy-norm}
    \norm{u}_{\eps} \coloneqq \bigl(B_{\eps}(u,u)\bigr)^{1/2}
    = \left(\eps^2 \norm{u'}_{L^2(\Omega)}^2 + \int_{\Omega} b\, u^2 \dif x\right)^{1/2}.
\end{equation}
In what follows we instead work with the equivalent $1,\eps$-norm introduced in \eqref{eq:energy-norm}.
Since $\underline{b}^2 \le b(x) \le \norm{b}_{L^{\infty}(\Omega)}$ on
$\overline{\Omega}$, \eqref{eq:appendix-energy-norm} and
\eqref{eq:energy-norm} are equivalent, \emph{uniformly in $\eps$}: 
\begin{equation}
    \label{eq:appendix-norm-equiv}
    \min\{1,\underline{b}\}\, \norm{u}_{1,\eps}
    \le \norm{u}_{\eps}
    \le \max\{1,\norm{b}_{L^{\infty}(\Omega)}^{1/2}\}\, \norm{u}_{1,\eps},
    \qquad u \in H^{1}(\Omega).
\end{equation}
For the solution $u_{\eps}$ of \eqref{eq:appendix-problem} one has the a priori estimate
\begin{equation*}
    \norm{u_{\eps}}_{1,\eps}
    \le \frac{1}{\min\{1,\underline{b}\}}
    \left(\underline{b}^{-1} \norm{f}_{L^2(\Omega)} + C\bigl(\abs{\alpha^{-}} + \abs{\alpha^{+}}\bigr)\right).
\end{equation*}

For the discretization we fix a finite-dimensional space $S_{0}
\subset H^{1}_{0}(\Omega)$ with $\dim S_{0} = N$ and a lifting $u_{0} \in
H^{1}_{D}(\Omega)$ of the boundary data, and replace $H^{1}_{D}(\Omega)$ and
$H^{1}_{0}(\Omega)$ by the affine space $S_{D} \coloneqq u_{0} + S_{0} \subset
H^{1}_{D}(\Omega)$ and by $S_{0}$, respectively.
The discrete solution $u_{FE} \in S_{D}$ is then characterized by
\begin{equation}
    \label{eq:appendix-discrete-problem}
    B_{\eps}(u_{FE},v) = F(v) \qquad \forall v \in S_{0}.
\end{equation}
By Galerkin orthogonality, the finite element solution $u_{FE}$
is the best approximant of the exact solution $u_{\eps}$ in the energy norm,
i.e.,
\begin{equation}
    \label{eq:appendix-quasi-optimality}
    \norm{u_{\eps} - u_{FE}}_{\eps} \le \inf_{v \in S_{D}} \norm{u_{\eps} - v}_{\eps}.
\end{equation}
Equivalently, by the norm equivalence \eqref{eq:appendix-norm-equiv}, one has
\begin{equation*}
    \norm{u_{\eps} - u_{FE}}_{1,\eps}
    \le \frac{\max\{1,\norm{b}_{L^{\infty}(\Omega)}^{1/2}\}}{\min\{1,\underline{b}\}}
    \inf_{v \in S_{D}} \norm{u_{\eps} - v}_{1,\eps}.
\end{equation*}

\subsection{Regularity of solutions}
In order to identify spaces $S_{D}$ in which the exact solution $u_{\eps}$ of
\eqref{eq:appendix-problem} can be approximated well, we first characterize the behavior of $u_{\eps}$ precisely. 
In this section we establish the necessary $\eps$-robust regularity estimates.
\begin{theorem}[{\cite[Theorem~1]{melenk1997boundarylayer}}]
    \label{thm:appendix-deriv-bounds}
    Let $u_{\eps}$ be the solution of \eqref{eq:appendix-problem}. Then there are constants $C,K > 0$ depending only on the right-hand side $f$, the coefficient $b$, and the boundary data $\alpha^{\pm}$ such that
    \begin{equation}
        \label{eq:appendix-deriv-bounds-ueps}
        \norm{u_{\eps}^{(n)}}_{L^2(\Omega)} 
        \le C K^n \max\{n^{\mfs_{x}} , \eps^{-1}\}^{n}, 
        \qquad \forall n \in \mathbb{N}_{0}. 
    \end{equation}
\end{theorem}
The estimate of \Cref{thm:appendix-deriv-bounds} captures the global growth of
the derivatives but reveals little about the boundary layers themselves.

Their behavior can be characterized by the classical asymptotic expansion of the solution of
\eqref{eq:appendix-problem}: for every $M \in \mathbb{N}_{0}$ the solution
$u_{\eps}$ can be decomposed into a smooth asymptotic part, two boundary layer terms, and a
remainder,
\begin{equation}
    \label{eq:appendix-decomposition}
    u_{\eps} = w_{M} + A_{M}^{-}u_{\eps}^{-} + A_{M}^{+} u_{\eps}^{+} + r_{M}.
\end{equation}
Its construction is classical. Substituting the formal series
$u_{\eps} \sim \sum_{j=0}^{\infty} \eps^{2j} u_{2j}$ into
\eqref{eq:appendix-problem} and matching equal powers of $\eps$ yields the
smooth part $w_{M}$ as the partial sum
\begin{equation*}
    w_M(x) \coloneqq \sum_{j=0}^{M} \eps^{2j} u_{2j},
\end{equation*}
where the terms $u_{2j}$ are defined recursively by
\begin{equation*}
    u_0(x) = \frac{f(x)}{b(x)}, \qquad u_{j+2}(x) \coloneqq \frac{1}{b(x)} u_{j}''(x), \qquad j = 0,2,4,\dots.
\end{equation*}

A direct computation gives $L_{\eps}(u_{\eps} - w_{M}) = \eps^{2M+2} u_{2M}''$,
which vanishes as $\eps \to 0$ for each fixed $M$. Thus $w_{M}$ solves the
differential equation in the asymptotic limit but in general fails to satisfy the
boundary conditions. This defect at the endpoints can be compensated by introducing two boundary
layer functions $u_{\eps}^{-}, u_{\eps}^{+}$, defined as the solutions of
\begin{equation}
    \label{eq:appendix-def-bl-terms-decomp}
    \begin{rcases}
        L_{\eps} u_{\eps}^- = 0 \quad \text{on } \Omega,\\
        u_{\eps}^-(-1) = 1,\, u_{\eps}^-(1)=0
    \end{rcases}
    \qquad
    \begin{rcases}
        L_{\eps} u_{\eps}^+ = 0 \quad \text{on } \Omega,\\
        u_{\eps}^+(-1) = 0,\, u_{\eps}^+(1)=1.
    \end{rcases}
\end{equation}
With the choice
\begin{equation*}
    A_{M}^{-} \coloneqq \alpha^{-} - w_{M}(-1), \qquad A_{M}^{+} \coloneqq \alpha^{+} - w_{M}(1),
\end{equation*}
the sum $w_{M} + A_{M}^{-} u_{\eps}^{-} + A_{M}^{+} u_{\eps}^{+}$ satisfies the
prescribed boundary values while still obeying the same differential equation as
$w_{M}$. The remainder $r_{M}$ is then fixed by requiring the
decomposition \eqref{eq:appendix-decomposition} to hold exactly, i.e. as the solution of
\begin{equation*}
    L_{\eps} r_{M} = \eps^{2M+2} u_{2M}'' \qquad \text{on } \Omega, \qquad r_{M}(\pm 1) = 0.
\end{equation*}
We now estimate each of the four terms in
\eqref{eq:appendix-decomposition} separately.

The following Gevrey recursion estimate replaces the holomorphic Cauchy
estimates of \cite[Lemma~2]{melenk1997boundarylayer}.
\begin{lemma}\label{lem:gevrey-recursion}
    Let $\mfs_{x} \ge 1$ and let $B,u_0 \in \cG^{\mfs_{x}}(\overline{\Omega})$ satisfy
    \begin{equation*}
        \norm{B^{(n)}}_{L^{\infty}(\Omega)} \leq C_B \gamma_B^n (n!)^{\mfs_{x}},
        \qquad
        \norm{u_0^{(n)}}_{L^{\infty}(\Omega)} \leq C_0 \gamma_0^n (n!)^{\mfs_{x}},
        \qquad n \in \mathbb{N}_0.
    \end{equation*}
    Define the functions $u_j$ recursively by
    \begin{equation*}
        u_{j+2}(x) \coloneqq B(x)u_j''(x), \qquad j=0,2,4,\dots.
    \end{equation*}
    Then there are constants $K_1,K_2>0$, depending only on
    $\mfs_{x},C_B,\gamma_B$, and $\gamma_0$, such that
    \begin{equation*}
        \norm{u_j^{(n)}}_{L^\infty(\Omega)} \le C_0 K_1^j K_2^n (j!)^{\mfs_{x}} (n!)^{\mfs_{x}}, \qquad j=0,2,4,\dots, \quad n \in \mathbb{N}_0.
    \end{equation*}
\end{lemma}

\begin{proof}
    Set
    \begin{equation*}
        \gamma \coloneqq \max\{\gamma_0,2\gamma_B,1\}
    \end{equation*}
    and choose $L>0$ such that
    \begin{equation*}
        L^2 \ge 2 C_B \gamma^2.
    \end{equation*}
    We first prove by induction on $j \in 2\mathbb{N}$ the stronger assertion
    \begin{equation}
        \label{eq:appendix-gevrey-recursion-intermediate}
        \norm{u_j^{(n)}}_{L^{\infty}(\Omega)} \le C_0 L^j \gamma^n \bigl((n+j)!\bigr)^{\mfs_{x}}, \qquad n \in \mathbb{N}_0.
    \end{equation}
    The assertion is immediate for $j=0$. Suppose that
    \eqref{eq:appendix-gevrey-recursion-intermediate} holds for some even $j$.
    Leibniz' formula gives
    \begin{equation*}
        u_{j+2}^{(n)} = \sum_{k=0}^n \binom{n}{k} B^{(k)} u_j^{(n-k+2)}.
    \end{equation*}
    Writing $N \coloneqq n+j+2$ and using the derivative bounds for $B$ and the induction
    hypothesis, we obtain
    \begin{equation*}
        \norm{u_{j+2}^{(n)}}_{L^\infty(\Omega)}
        \le C_BC_0L^j\gamma^{n+2} \sum_{k=0}^n \left(\frac{\gamma_B}{\gamma}\right)^k \binom{n}{k}(k!)^{\mfs_{x}}\bigl((N-k)!\bigr)^{\mfs_{x}}.
    \end{equation*}
    Since $n\leq N$ and $\mfs_{x}\geq1$,
    \begin{equation*}
        \binom{n}{k}(k!)^{\mfs_{x}}\bigl((N-k)!\bigr)^{\mfs_{x}}
        \leq
        \binom{N}{k}(k!)^{\mfs_{x}}\bigl((N-k)!\bigr)^{\mfs_{x}}
        \leq
        (N!)^{\mfs_{x}}.
    \end{equation*}
    Consequently,
    \begin{equation*}
        \norm{u_{j+2}^{(n)}}_{L^\infty(\Omega)}
        \leq
        2C_B\gamma^2C_0L^j\gamma^n
        \bigl((n+j+2)!\bigr)^{\mfs_{x}}
        \leq
        C_0L^{j+2}\gamma^n
        \bigl((n+j+2)!\bigr)^{\mfs_{x}}.
    \end{equation*}
    This proves \eqref{eq:appendix-gevrey-recursion-intermediate}.
    
    Finally,
    \begin{equation*}
        \bigl((n+j)!\bigr)^{\mfs_{x}} = \left(\binom{n+j}{n}n!j!\right)^{\mfs_{x}} \le 2^{\mfs_{x}(n+j)}(n!)^{\mfs_{x}}(j!)^{\mfs_{x}}.
    \end{equation*}
    The result now follows with
    \begin{equation*}
        K_1 \coloneqq 2^{\mfs_{x}}L, \qquad K_2 \coloneqq 2^{\mfs_{x}} \gamma.
    \end{equation*}
\end{proof}
The Gevrey recursion estimate \cref{lem:gevrey-recursion} allows us to control
the growth of the derivatives of the asymptotic part $w_{M}$.
\begin{theorem}[{\cite[Theorem~3]{melenk1997boundarylayer}}]
    \label{thm:deriv-bound-smooth-part}
    There are constants $C, K_{1}, K_{2} > 0$ depending only on the input data $f$, $b$, $\mfs_{x}$, 
    and $\alpha^{\pm}$ such that the functions $w_{M}$ of \eqref{eq:appendix-decomposition} 
    satisfy the following estimate: 
    Under the assumption $0 < {\eps K_{1} (2 M)^{\mfs_{x}}} \le 1$
    \begin{equation*}
        \norm{w_{M}^{(n)}}_{L^{\infty}(\Omega)} \le C K_{2}^{n} (n!)^{\mfs_{x}}, \qquad \forall n \in \mathbb{N}_{0}.
    \end{equation*}
\end{theorem}
\begin{proof}
    This follows by applying \cref{lem:gevrey-recursion} to the terms $u_{2j}$, exactly as in the proof of \cite[Theorem~3]{melenk1997boundarylayer}. Note that, since $b \in \cG^{\mfs_{x}}(\overline\Omega)$ is bounded away from $0$, its reciprocal $B(x) \coloneqq 1/b(x)$ again belongs to $\cG^{\mfs_{x}}(\overline\Omega)$.
\end{proof}
A direct consequence is that the coefficients $A_{M}^{-}$ and $A_{M}^{+}$ stay bounded.
\begin{corollary}[{\cite[Corollary~4]{melenk1997boundarylayer}}]
    With $K_{1}$ as in \cref{thm:deriv-bound-smooth-part} there is $C>0$ depending only on the input data $f,b$, the Gevrey index $\mfs_{x}$, and $\alpha^{\pm}$ such that for any $\eps, M$ 
    satisfying $0 < \eps K_{1} (2 M)^{\mfs_{x}} \le 1$ 
    \begin{equation*}
        \abs{A_{M}^{-}}, \abs{A_{M}^{+}} \le C.
    \end{equation*}
\end{corollary}
We now turn to the two boundary layer terms.
\begin{theorem}[{\cite[Theorem~5]{melenk1997boundarylayer}}]
    Let $u_{\eps}^{-}, u_{\eps}^{+}$ be solutions of \eqref{eq:appendix-def-bl-terms-decomp}. Then there are $C, K_{3} > 0$ depending only on the function $b$ such that
    \begin{align}
        \abs{(u_{\eps}^{-})^{(n)}(x)} &\le e^{- \underline{b}(x+1)/\eps} C K_{3}^{n} \max\{n^{{\mfs_{x}}}, \eps^{-1}\}^{n} \qquad \forall x \in \overline{\Omega}, \quad n \in \mathbb{N}_{0},\\
        \abs{(u_{\eps}^{+})^{(n)}(x)} &\le e^{- \underline{b}(1-x)/\eps} C K_{3}^{n} \max\{n^{{\mfs_{x}}}, \eps^{-1}\}^{n} \qquad \forall x \in \overline{\Omega}, \quad n \in \mathbb{N}_{0}.
    \end{align}
\end{theorem}
Finally, we consider the remainder $r_{M}$.
\begin{theorem}[{\cite[Theorem~6]{melenk1997boundarylayer}}]
    There are constants $C, K_{4} > 0$ depending only on the functions $f, b$ and on $\mfs_{x}$
    such that the remainders $r_{M}$ defined in \eqref{eq:appendix-decomposition} satisfy
    \begin{equation*}
        \norm{r_{M}^{(n)}}_{L^2(\Omega)} \le C \eps^{2-n} (\eps K_{4} (2 M)^{\mfs_{x}})^{2M}, \qquad n = 0,1,2.
    \end{equation*}
\end{theorem}
\begin{remark}
    \Cref{thm:deriv-bound-smooth-part} and the preceding theorem show that the
    asymptotic expansion is informative precisely in the regime
    $\eps \max\{K_{1}, K_{4}\} (2M)^{\mfs_{x}} < 1$, 
    i.e.\ while $2M$ remains small compared with $\eps^{-1/\mfs_{x}}$; beyond it the expansion ceases to be useful.
\end{remark}
\subsection{Polynomial approximation}
The aim of this section is to show that, for the $H^{1}$-conforming
approximation with piecewise polynomials, it is enough to control the growth of
the derivatives on each element.
Let $I = [-1,1]$. For a bounded interval $J \subset \mathbb{R}$ and $p \in \mathbb{N}_{0}$, let $\Pi_{p}(J)$ denote the space of polynomials on $J$ of degree at most $p$.

In \cite[Appendix~A]{costabel2005maxwell} the following $C^{k-1}$-conforming 
and univariate projector has been constructed. 
Here we use the presentation from \cite[Lemma~2]{feischl2020exponential}.
\begin{lemma}
    \label{lem:hp-projector}
    For any $p,k \in \mathbb{N}$ with $p \ge 2k-1$, there is a projector $\hat{\pi}_{p,k} \colon H^{k}(I) \to \Pi_{p}(I)$ that satisfies $(\hat{\pi}_{p,k} u)^{(k)} = \hat{\pi}_{p-k,0}(u^{(k)})$, and $(\hat{\pi}_{p,k} u)^{(j)} (\pm 1) = u^{(j)}(\pm 1)$, for any $j=0,\dots,k-1$.
    
    Moreover, there holds:
    \begin{enumerate}[(i)]
        \item For every $k \in \mathbb{N}$, there exists a constant $C_{k} > 0$ such that
        \begin{equation*}
            \forall u \in H^{k}(I),\, \forall p \ge 2k-1 : \quad \norm{\hat{\pi}_{p,k} u}_{H^{k}(I)} \le C_{k} \norm{u}_{H^{k}(I)}.
        \end{equation*}
        \item For integers $p,k \in \mathbb{N}$ with $p \ge 2k-1$, $k_{1} = p - k + 1$ and for $u \in H^{k+r}(I)$ with any $k \le r \le k_{1}$ there holds the error bound
        \begin{equation*}
            \norm{(u - \hat{\pi}_{p,k} u)^{(j)}}_{L^2(I)}^2 \le \frac{(k_{1} - r)!}{(k_{1} + r)!} \norm{u^{(k+r)}}_{L^2(I)}^{2}, \qquad j=0,1,\dots,k.
        \end{equation*}
    \end{enumerate}
\end{lemma}
The following is the $\mfs_{x}$-Gevrey analogue of \cite[Lemma~9]{melenk1997boundarylayer}.  Its proof differs from the analytic one, which relies on approximation by Chebyshev interpolation of a holomorphic extension; we instead argue directly with the projector $\hat{\pi}_{p,2}$ of \cref{lem:hp-projector} and the Gevrey derivative bounds.
\begin{lemma}
    \label{lem:gevrey-poly-approx}
    Let $\mfs_{x} \ge 1$, and let $u\in C^\infty(I)$ satisfy
    \begin{equation*}
        \norm{u^{(n)}}_{L^2(I)}
        \leq C_u \gamma^n (n!)^{\mfs_{x}}
        \qquad \forall n\in\mathbb N_0
    \end{equation*}
    for some $C_u,\gamma>0$. 
    Then there is a sequence of polynomials
    $(P_p)_{p=0}^\infty$, $P_p\in{\Pi_p}(I)$, such that
    \begin{equation*}
        \norm{u-P_p}_{L^\infty(I)}
        +
        \norm{(u-P_p)'}_{L^\infty(I)}
        \leq
        C C_u \exp\bigl(-\sigma p^{1/\mfs_{x}}\bigr),
    \end{equation*}
    where $C,\sigma>0$ depend only on $\gamma$ and $\mfs_{x}$.
    For $p\geq3$, one may take $P_p=\hat{\pi}_{p,2}u$.
\end{lemma}
\begin{proof}
    Let $p \ge 3$, set $P_p=\hat{\pi}_{p,2}u$, and let
    $2\leq r\leq p-1$. Summing the estimates of \cref{lem:hp-projector}(ii) for $j=0,1,2$ yields
    \begin{equation*}
        \norm{u-P_p}_{H^2(I)}
        \leq
        C
        \left(
        \frac{\Gamma(p-r)}
        {\Gamma(p+r)}
        \right)^{1/2}
        \norm{u^{(r+2)}}_{L^2(I)}.
    \end{equation*}
    Using the recursion formula for the Gamma function, we obtain
    \begin{equation*}
        \left(
        \frac{\Gamma(p-r)}
        {\Gamma(p+r)}
        \right)^{1/2}
        \leq
        (p-r)^{-r}.
    \end{equation*}
    Hence, by the Gevrey estimate and Stirling's inequality,
    \begin{equation*}
        \norm{u-P_p}_{H^2(I)}
        \leq
        C C_u
        \frac{\gamma^{r+2}((r+2)!)^{\mfs_{x}}}
        {(p-r)^r}
        \leq
        C C_u
        \left(
        \frac{\widetilde\gamma r^{\mfs_{x}}}
        {p-r}
        \right)^r
    \end{equation*}
    for some $\widetilde\gamma>0$ depending only on $\gamma$ and $\mfs_{x}$.
    
    Choose $c>2\widetilde\gamma+1$ and set
    \begin{equation*}
        r=\left\lfloor (p/c)^{1/\mfs_{x}}\right\rfloor.
    \end{equation*}
    For all sufficiently large $p$, this gives $2\leq r<p$ and
    \begin{equation*}
        \frac{\widetilde\gamma r^{\mfs_{x}}}{p-r}
        \leq
        \frac{\widetilde\gamma (p/c)}
        {p(1-1/c)}
        = \frac{\widetilde\gamma}{c-1} \eqqcolon q<1.
    \end{equation*}
    %
    Therefore,
    \begin{equation*}
        \norm{u-P_p}_{H^2(I)}
        \leq
        C C_u q^r
        \le C C_u q^{-1}
        \exp\left(
        -\frac{|\log q|}{c^{1/\mfs_{x}}}p^{1/\mfs_{x}}
        \right).
    \end{equation*}
    Absorbing the factor $q^{-1}$ into $C$ and setting
    \begin{equation*}
        \sigma \coloneqq \frac{|\log q|}{c^{1/\mfs_{x}}}>0,
    \end{equation*}
    we conclude that
    \begin{equation*}
        \norm{u-P_p}_{H^2(I)}
        \leq
        C C_u\exp\bigl(-\sigma p^{1/\mfs_{x}}\bigr).
    \end{equation*}
    The assertion follows from the embedding
    $H^2(I)\hookrightarrow W^{1,\infty}(I)$. The finitely many remaining
    values of $p$ are absorbed into the constant $C$.
\end{proof}

For a polynomial degree $p \ge 1$, 
we denote by $i_p: C(I)\to \Pi_p(I)$ 
the interpolation operator associated with the $p+1$ Gauss--Lobatto points 
in $I$.
There exists $C_{G} > 0$ such that
\begin{equation}
    \label{eq:appendix-stability-gauss-lobatto}
    \forall u\in C(I) \; \forall p\geq 1\;: \quad 
    \norm{i_{p}u}_{L^{\infty}(I)} \le C_{G} (1 + \ln p) \norm{u}_{L^{\infty}(I)}.
\end{equation}
For a proof, see \cite{sundermann1983lebesgue}.

\begin{lemma}[{\cite[Lemma~10]{melenk1997boundarylayer}}]
    Let $u \in C^{1}(I)$. 
    Then there holds with the constant 
    $C_G$ from \eqref{eq:appendix-stability-gauss-lobatto} for all degrees $p\geq 1$
    \begin{align*}
        \norm{u - i_{p}u}_{L^{\infty}(I)} 
        &\le (1 + C_{G}) (1 + \ln p) \norm{u}_{L^{\infty}(I)},\\
        \norm{(u - i_{p}u)'}_{L^{\infty}(I)}
        &\le \norm{u'}_{L^{\infty}(I)} + C_{G} (1 + \ln p) p^2 \norm{u}_{L^{\infty}(I)}.
    \end{align*}
\end{lemma}

The error of the Gauss--Lobatto interpolation satisfies the following estimate:
\begin{lemma}
    Let $u$ satisfy the assumptions of \cref{lem:gevrey-poly-approx}. 
    Then there are $C, \sigma > 0$ depending only on $\gamma$ of \cref{lem:gevrey-poly-approx}, $\mfs_{x}$ and $C_{G}$ such that
    \begin{equation*}
        \norm{u - i_{p}u}_{L^{\infty}(I)} + \norm{(u - i_{p}u)'}_{L^{\infty}(I)} \le C C_{u} e^{-\sigma p^{1/\mfs_{x}}}.
    \end{equation*}
\end{lemma}
\begin{remark}
    Because $\pm1$ belong to the Gauss--Lobatto nodes, 
    the interpolant reproduces the endpoint values, $u(\pm1) = i_{p}u(\pm1)$.
\end{remark}
\subsection{Robust approximation on a three-element mesh}
\label{sec:RobApp3El}
Following \cite[Section~4]{melenk1997boundarylayer}, we introduce the following
piecewise-polynomial spaces and three-element mesh.
\begin{definition}
    Let $\Delta = \{x_i : 0 \le i \le n\}$ be 
    a collection of $n+1 \ge 2$ nodal points satisfying
    \begin{equation*}
        -1 = x_{0} < x_{1} < \cdots < x_{n} = 1.
    \end{equation*}
    These $n+1$ nodal points define a mesh on $\Omega$, 
    i.e., a partition of $\Omega$ into open ``elements'' $I_{i} := (x_{i-1},x_{i})$, $i=1,\dots,n$.
    On this mesh define the spaces of continuous piecewise polynomials of degree at most $p$ by
    \begin{align*}
        S^p(\Delta) &\coloneqq \{u \in H^{1}(\Omega) : u|_{I_{i}} \in \Pi_{p}(I_{i}), i=1,\dots,n\},\\
        S_{0}^p(\Delta) &\coloneqq S^p(\Delta) \cap H_{0}^{1}(\Omega).
    \end{align*}
\end{definition}
\begin{definition}
    \label{def:appendix-approx-spaces}
    For $\kappa > 0$ and $p \in \mathbb{N}$,
    introduce spaces $S(\kappa,p) \subset H^1(\Omega)$ 
    of continuous piecewise polynomials of degree at most $p$ by
    \begin{equation*}
        S(\kappa, p) \coloneqq 
        \left\{ 
        \begin{alignedat}{2}
            S^{p}(\Delta), \quad \Delta &= \{-1, 1\} \quad &&\text{for } \kappa p \eps \ge 1/2\\
            S^{p}(\Delta), \quad \Delta &= \{-1, -1 + \kappa p \eps, 1 - \kappa p \eps, 1\} \quad &&\text{for } \kappa p \eps < 1/2.
        \end{alignedat}
        \right.
    \end{equation*}
\end{definition}
With this, we are now able to state the Gevrey analogue of \cite[Theorem~16]{melenk1997boundarylayer}.
\begin{theorem}
    \label{thm:gevrey-three-element}
    Let $\mfs_{x} \ge 1$ and let
    $f,b\in \cG^{\mfs_{x}}(\overline{\Omega})$, with $b \ge \underline{b}^2$ on
    $\overline{\Omega}$. Let $u_\eps$ be the solution of
    \eqref{eq:appendix-problem}. Then there are $C,\tau,\kappa_0>0$ depending
    only on $f$, $b$, $\underline{b}$, $\alpha^\pm$, and $\mfs_{x}$ such that the
    following holds. For each $0<\kappa\leq\kappa_0$ and $p\in\mathbb{N}$
    there is $v_p\in S(\kappa,p)$ with $v_p(\pm1)=u_\eps(\pm1)$ and
    \begin{equation*}
        \norm{u_\eps-v_p}_{L^{\infty}(\Omega)} + \kappa p \eps \norm{(u_\eps-v_p)'}_{L^{\infty}(\Omega)}
        \le C(1+\ln p)p^2 \exp\left(-\tau (\kappa p)^{1/\mfs_{x}}\right).
    \end{equation*}
    In particular, for $\mfs_{x}=1$ the
    estimate reduces to the robust exponential rate
    $Ce^{-\tau \kappa p}$ of the analytic case.
\end{theorem}
\begin{corollary}[{\cite[Corollary~18]{melenk1997boundarylayer}}]
    \label{cor:gevrey-hp-galerkin}
    In \eqref{eq:appendix-discrete-problem} set $S_{0} \coloneqq S(\kappa,p) \cap H_{0}^{1}(\Omega)$ and $S_{D} \coloneqq l(x) + S_{0}$ with $l(x) = \alpha^{-}(1-x)/2 + \alpha^{+} (x+1)/2$. Denote by $u_{FE}$ the finite element solution of \eqref{eq:appendix-discrete-problem}. Then there is $\kappa_{0}$ depending only on the input data $f,b$, $\alpha^{\pm}$ and $\mfs_{x}$ such that for all $0 < \kappa \le \kappa_{0}$ there are $C, \sigma > 0$ independent of $\eps, p$ such that
    \begin{equation*}
        \norm{u_{\eps} - u_{FE}}_{1,\eps} \le C e^{- \sigma p^{1/\mfs_{x}}}.
    \end{equation*}
\end{corollary}
\begin{remark}
    \label{rem:appendix-uniform-hp-constants}
    For fixed $f$, $\alpha^\pm$, and $\mfs_{x}$, the dependence of $\kappa_0$,
    and of $C$ and $\sigma$ for fixed admissible $\kappa$, on $b$ is only
    through its positive lower and upper bounds and the constants quantifying
    its spatial Gevrey regularity. Consequently, if these bounds are uniform
    for a family $b(\cdot;\bm y)$, then $\kappa_0$, $C$, and $\sigma$ can be
    chosen uniformly in $\bm y$.
\end{remark}

\subsection{Impact of numerical quadrature}
If the coefficient $b$ and the right-hand side $f$ are general $\mfs_{x}$-Gevrey
functions, any implementation of \eqref{eq:appendix-discrete-problem} must resort to numerical quadrature for the evaluation of the stiffness matrix and
the load vector. We show that employing a composite Gaussian quadrature rule with $p+q$ quadrature points on each subinterval, where $q$ is proportional to $p$, is sufficient to preserve the robust 
root-exponential convergence of the finite element method based on the spaces $S(\kappa,p)$. 

For brevity, we focus on the case with homogeneous Dirichlet
boundary conditions $\alpha^{\pm} = 0$. For an interval $J$, denote by $G_{J}^{q}(w)$ the $q$-point Gaussian quadrature rule on $J$ applied to the function $w$. 
Replacing the integrals in \eqref{eq:appendix-discrete-problem} by a $p+q$-point Gaussian quadrature rule in each element leads to the problem:
\begin{equation}
    \label{eq:appendix-sol-with-quadrature}
    \text{find } u_{G} \in S_{0}^{p}(\Delta) \text{ such that} \qquad B_{\eps}^{p+q}(u_{G},v) = F^{p+q}(v), \qquad \forall v \in S_{0}^{p}(\Delta).
\end{equation}
Here the bilinear form $B_{\eps}^{p+q}$ and the right-hand side $F^{p+q}$ are given by
\begin{equation*}
    B_{\eps}^{p+q}(u,v) = \sum_{i=1}^{n} \left(\eps^2 G_{I_{i}}^{p+q}(u' v') + G_{I_{i}}^{p+q}(b u v) \right), \qquad F^{p+q}(v) = \sum_{i=1}^{n} G_{I_{i}}^{p+q}(fv),
\end{equation*}
and $S_{0}^{p}(\Delta) = S(\kappa,p) \cap H_{0}^{1}(\Omega)$ are the approximation spaces introduced in \ref{def:appendix-approx-spaces}.

\begin{theorem}
    \label{thm:appendix-quadrature}
    Let $f,b \in \cG^{\mfs_{x}}(\overline{\Omega})$, let $u_{\eps}$ be the solution of \eqref{eq:appendix-problem} with $\alpha^{\pm} = 0$. Let $u_{G}$ be defined by \eqref{eq:appendix-sol-with-quadrature}. Then there are constants $C, \sigma > 0$ depending only on the input data $f,b$ and $\mfs_{x}$ such that for all $p,q \in \mathbb{N}$
    \begin{equation*}
        \norm{u_{\eps} - u_{G}}_{1,\eps} \le C \left(\inf_{v \in S_{0}^{p}(\Delta)}\norm{u_{\eps} - v}_{1,\eps} + e^{-\sigma q^{1/\mfs_{x}}} \right).
    \end{equation*}
\end{theorem}

\end{document}